\documentclass[a4paper,10pt]{article}

\usepackage{lmodern}
\usepackage[T1]{fontenc}
\usepackage[english]{babel}
\usepackage[utf8]{inputenc}

\usepackage{amsmath}
\usepackage{amssymb}
\usepackage{amsthm}
\usepackage{mathtools}
\usepackage{mathrsfs}
\usepackage{abraces}
\usepackage{bbm}

\usepackage{tikz,tikz-3dplot}
\usetikzlibrary{cd,patterns,positioning,decorations.markings,decorations.pathmorphing,decorations.pathreplacing}
\usepackage{xstring}
\usepackage{ifthen}
\usepackage{pdflscape}
\usepackage{diagbox}
\usepackage[labelfont=bf,font=small]{caption}

\usepackage{titlesec}
\titleformat{\section}{\normalfont\fillast\scshape}{\thesection.}{.5em}{}
\titleformat{\subsection}[runin]{\normalfont\bfseries}{\thesubsection.}{.5em}{}[.]

\usepackage{enumitem}

\setlist[description]{%
	font={\normalfont\itshape},
	}
\renewcommand\labelenumi{(\textit{\roman{enumi}})}
\renewcommand\theenumi\labelenumi

\usepackage[autostyle]{csquotes}
\usepackage[backend=biber,style=numeric,doi=false,isbn=false,url=false,giveninits=true,maxbibnames=99]{biblatex}
\newtheorem{thm}{Theorem}[subsection]
\newtheorem{lem}[thm]{Lemma}
\newtheorem{pro}[thm]{Proposition}

\newtheorem*{thm*}{Theorem}
\newtheorem*{lem*}{Lemma}
\newtheorem*{pro*}{Proposition}
\newtheorem*{cor*}{Corollary}

\newtheorem{cla*}{Claim}
\theoremstyle{definition}
\newtheorem{dfn}[thm]{Definition}
\newtheorem*{dfn*}{Definition}
\theoremstyle{remark}
\newtheorem{rmk}[thm]{Remark}
\newtheorem{exa}[thm]{Example}
\newtheorem*{rmk*}{Remark}
\newtheorem*{rmks*}{Remarks}
\newtheorem*{exa*}{Example}

\DeclareMathOperator{\Hom}{Hom}
\DeclareMathOperator{\End}{End}
\DeclareMathOperator{\id}{id}
\DeclareMathOperator{\JW}{JW}

\DeclareMathOperator{\Proj}{Proj}

\newcommand{\cohdeg}[1]{\scriptstyle\textcolor{green}{#1}}

\newcommand{\Ao}{A^\bullet}
\newcommand{\Bo}{B^\bullet}
\newcommand{\C}{\mathcal{C}}
\newcommand{\Cb}{\mathcal{C}^{\mathrm{b}}}
\newcommand{\CC}{\mathscr{C}}

\newcommand{\hh}{\mathfrak{h}}
\newcommand{\HH}{\mathbf{H}}
\newcommand{\ZZ}{\mathbb{Z}}

\newcommand{\D}{\mathscr{H}}
\newcommand{\DBS}{\D_{\mathrm{BS}}}

\newcommand{\Kb}{\mathcal{K}^{\mathrm{b}}}
\newcommand{\K}{\mathcal{K}}

\newcommand{\un}{\mathbbm{1}}

\newcommand{\boi}{\mathbf{i}}
\newcommand{\boj}{\mathbf{j}}

\newcommand{\uw}{\underline{w}}
\newcommand{\uz}{\underline{z}}

\newcommand{\ux}{\underline{x}}

\newcommand{\uom}{\underline{\omega}}

\newcommand{\kk}{\Bbbk}

\newcommand{\Roudg}[1]{F_{#1}^\bullet}

\newcommand{\Laur}{\ZZ[v,v^{-1}]}

\newcommand{\cword}{S^*}
\newcommand{\bword}{\Sigma^*}

\newcommand{\chr}{\mathrm{char}}
\newcommand{\stdmix}{\Delta^{\mathrm{mix}}}
\newcommand{\cosmix}{\nabla^{\mathrm{mix}}}

\newcommand{\gr}[1]{\textcolor{green}{#1}}
\newcommand{\re}[1]{\textcolor{red}{#1}}
\newcommand{\bl}[1]{\textcolor{blue}{#1}}

\newcommand{\worsum}[1]{C_{#1}}

\definecolor{green}{rgb}{0,.5,0}

\tikzset{%
	virtual/.style  = {font=\scriptsize,draw=black,circle,dashed,inner sep=2pt},
	real/.style     = {circle,inner sep=2pt,font=\scriptsize},
	patch/.style    = {gray,pattern=north west lines, pattern color=gray},
	positive/.style = {line width=2pt},
	negative/.style = {line width=.25pt,double distance=1.5pt},
	braid/.style = {thick,double distance=2pt,rounded corners},
}

\tikzset{%
	pics/dot/.style n args={1}{%
		code={%
			\draw (0,0) -- (0,#1);
			\fill (0,#1) circle (1.5pt); 		
			}
	}
}
\def\d{.5cm}\def\h{1cm}\def\hdot{.4cm}

\colorlet{s}{red}
\colorlet{t}{blue}
\colorlet{u}{green}
\colorlet{v}{yellow}
\colorlet{g}{violet}
\colorlet{dgdegree}{green!40!black}

\tikzset{%
	pics/topdot/.style n args={3}{%
		code= {%
			\StrLen{#1}[\length]%
			\def \d{0.35}%
			\def \h{1}%
			\node[anchor=east,inner sep=0pt] at ({-(\length+1)*\d/2},0) {$#3$};%
			\foreach \i in {1,...,\length}{%
				\StrChar{#1}{\i}[\color]%
				\ifthenelse{\equal{\color}{O}}{%
					\draw[->] ({-(\length+1)*\d/2+\i*\d},-\h/2)--({-(\length+1)*\d/2+\i*\d},{1/15*\h/2});%
					\draw ({-(\length+1)*\d/2+\i*\d},{1/15*\h/2})--({-(\length+1)*\d/2+\i*\d},\h/2);%
				}%
				{%
					\ifthenelse{\equal{\color}{o}}{%
						\draw ({-(\length+1)*\d/2+\i*\d},-\h/2)--({-(\length+1)*\d/2+\i*\d},{-1/15*\h/2});%
						\draw[<-] ({-(\length+1)*\d/2+\i*\d},{-1/15*\h/2})--({-(\length+1)*\d/2+\i*\d},\h/2);%
					}%
					{%
						\ifthenelse{\i=#2}{%
							\draw[draw=\color] ({-(\length+1)*\d/2+\i*\d},-\h/2)--({-(\length+1)*\d/2+\i*\d},0);%
							\fill[fill=\color] ({-(\length+1)*\d/2+\i*\d},0) circle (2pt);%
						}%
						{%
							\draw[draw=\color] ({-(\length+1)*\d/2+\i*\d},-\h/2)--({-(\length+1)*\d/2+\i*\d},\h/2);%
						}%
					}%
				}%
			}%
		}%
	}
}

\tikzset{%
	pics/bottomdot/.style n args={3}{%
		code= {%
			\StrLen{#1}[\length]%
			\def \d{0.35}%
			\def \h{1}%
			\ifthenelse{\equal{\length}{1}}{%
					\node[anchor=east,inner sep=0pt] at ({-(\length)*\d/2},{.2*\h}) {$#3$};%
				}%
				{%
					\node[anchor=east,inner sep=0pt] at ({-(\length+1)*\d/2},0) {$#3$};%
				}
			\foreach \i in {1,...,\length}{%
				\StrChar{#1}{\i}[\color]%
				\ifthenelse{\equal{\color}{O}}{%
					\draw[->] ({-(\length+1)*\d/2+\i*\d},-\h/2)--({-(\length+1)*\d/2+\i*\d},{1/15*\h/2});%
					\draw ({-(\length+1)*\d/2+\i*\d},{1/15*\h/2})--({-(\length+1)*\d/2+\i*\d},\h/2);%
				}%
				{%
					\ifthenelse{\equal{\color}{o}}{%
						\draw ({-(\length+1)*\d/2+\i*\d},-\h/2)--({-(\length+1)*\d/2+\i*\d},{-1/15*\h/2});%
						\draw[<-] ({-(\length+1)*\d/2+\i*\d},{-1/15*\h/2})--({-(\length+1)*\d/2+\i*\d},\h/2);%
					}%
					{%
						\ifthenelse{\i=#2}{%
							\draw[draw=\color] ({-(\length+1)*\d/2+\i*\d},\h/2)--({-(\length+1)*\d/2+\i*\d},0);%
							\fill[fill=\color] ({-(\length+1)*\d/2+\i*\d},0) circle (2pt);%
						}%
						{%
							\draw[draw=\color] ({-(\length+1)*\d/2+\i*\d},-\h/2)--({-(\length+1)*\d/2+\i*\d},\h/2);%
						}%
					}%
				}%
			}%
		}%
	}
}
\title{\bfseries Reducing Rouquier complexes}
\author{Leonardo Maltoni}
\begin{document}
	\maketitle
	\begin{abstract}
		We describe reduced representatives for positive 
		Rouquier complexes.
		These are obtained via \emph{Morse theoretical} Gaussian elimination from the corresponding 
		standard representatives. The underlying graded object is still a direct sum of
		Bott-Samelson objects, but over subwords rather than subexpressions.
	\end{abstract}
	\section*{Introduction}
			Rouquier complexes were introduced in \cite{Rou_cat} to study actions of the (generalized)
			braid group on categories. 
			
			These actions appeared in representation theory already in the works of Carlin
			\cite{Car} or Rickard \cite{Rick}, and 
			a precise definition was made by Deligne \cite{Deli}.
			They describe higher symmetries 
			of the categories acted upon, usually giving rich information about them. A classical example is the 
			action of the braid group $B_W$ on the bounded derived category $\mathcal{O}$ associated to a complex semisimple Lie algebra $\mathfrak{g}$, with Weyl group $W$.
			A geometric counterpart is the action on the bounded derived category of 
			sheaves over the flag variety $G/B$ associated with a reductive algebraic group $G$ and 
			a Borel subgroup $B$.
			
			Rouquier pointed out the interest of the category of self-equivalences 
			induced by such an action. So he introduced the 
			\textit{2-braid group} $\mathscr{B}_W$, that upgrades $B_W$, for any Coxeter system $(W,S)$, to a category and serves as a 
			model for this purpose.
			It is a subcategory of $\K(\D)$, the homotopy category of the \emph{Hecke category} $\D$, and its 
			objects, called \emph{Rouquier complexes}, are all possible tensor products between certain 
			\emph{standard} and \emph{costandard} complexes $F_s$ and $F_s^{-1}$, for $s\in S$. 
			Such products satisfy, up to homotopy, the relations of the braid group 
			so their isomorphism classes
			form at least a quotient of $B_W$. See Remark \ref{rmk_faithful} below for more details.
			
			This category also plays an important role in algebraic topology.
			Khovanov and Rozansky \cite{KhovRoz} and then Khovanov \cite{Khov} constructed 
			a triply graded link homology from Rouquier complexes.
			The idea is that the homology
			of the complex obtained computing Hochschild (co)homology of the Rouquier complex $F_\omega$
			associated to a braid $\omega$, is an invariant, up to an overall shift, of the link $\overline{\omega}$ obtained by closing $\omega$.

			When working with Rouquier complexes, it is useful to consider appropriate representatives 
			in the category $\C(\D)$ of complexes. The natural ones are just the tensor products between 
			the standard and costandard complexes. 
			When the category $\D$ has a complete local coefficient ring 
			(see Theorem \ref{thm_soergcat}), any complex 
			admits a \emph{minimal subcomplex}: a summand which is homotopy equivalent to the original complex and
			has no contractible summand. This is very hard to 
			find in general.
			In this	paper we find an intermediate reduction for positive (or negative) Rouquier complexes 
			(i.e.\ with only $F_s$'s, or, resp., only $F_s^{-1}$'s) which is 
			available with no restriction on the coefficients. 

			As a first application, this result has been used in \cite{Mal_wak} to study the extension groups between 
			Wakimoto sheaves in the homotopy category of the affine Hecke category in type $\tilde{A_1}$
			over $\mathbb{Z}$.
			\subsection{Diagrammatic Hecke category}
			As we mentioned, Rouquier complexes are certain objects of the homotopy 
			category of the Hecke category $\D$. The original definition was in terms of Soergel bimodules. 
			In this paper we will use the diagrammatic presentation of $\D$ due to Elias 
			and Williamson \cite{EW}.
			 
			The generating objects, denoted $B_s$, correspond to simple reflections, and are 
			represented as colored points (one color for each $s\in S$). 
			Products between these generators are called \emph{Bott-Samelson} objects. They correspond
			to words and are represented by sequences of colored points on a line.
			The morphisms between 
			Bott-Samelson objects are described by certain \emph{diagrams} inside the strip
			$\mathbb{R}\times [0,1]$. Roughly, these are equivalence classes (with respect to certain relations) of
			planar graphs, obtained from some generating vertices, that connect the sequences of 
			points corresponding to the source and the target.
			This \emph{diagrammatic} Hecke category is available for any Coxeter system, equipped with a \emph{realization} (i.e.\ a 
			finite rank representation over the coefficient ring satisfying certain properties).
			For more details see \S \ref{sec_diagHeck} below. Under certain conditions this category is equivalent to 
			that of Soergel bimodules and has a geometric counterpart, when $W$ is a Weyl group, given by 
			equivariant parity sheaves on the corresponding flag variety.

			Let us pass to the homotopy category $\K(\D)$. For any $s\in S$, the standard and costandard 
			complexes $F_s$ and $F_{s}^{-1}$ are defined as:
			\begin{equation}\label{eq_stdcostd}
				\begin{tikzcd}[row sep=tiny, column sep=0.7cm,%
								execute at end picture={%
									\def\hdott{.32cm}
									\pic[red,yshift=.3em] at (A) {dot={\hdott}};
									\pic[red,yshift=1.5em,yscale=-1] at (B) {dot={\hdott}};
									}]
					F_{\re{s}}=\,\cdots\ar[r]		& 0 \ar[r]		& 0	\ar[r]								& B_{\re{s}}\ar[r,""{coordinate,name=A}]	& \un(1)\ar[r]	& 0\ar[r]		& \cdots \\
																& \cohdeg{-2}	& \cohdeg{-1}							& \cohdeg{0}											& \cohdeg{1}	& \cohdeg{2}	& \\
					F_{\re{s}}^{-1}=\,\cdots\ar[r]	& 0\ar[r]		& \un(-1)\ar[r,""{coordinate,name=B}]	& B_{\re{s}}\ar[r]							& 0	\ar[r]		& 0\ar[r]		& \cdots \\
				\end{tikzcd}			
			\end{equation}
			where the only non-zero morphisms are represented by their diagrams in $\D$ and
			the numbers in the middle row denote the cohomological degree in $\K(\D)$. 
			Rouquier complexes are all the possible tensor products between the $F_{s}^{\pm 1}$'s.
			\subsection{Results}
			We consider a \emph{positive} Rouquier complex, of the form:
			\[
				F_{\uw}^\bullet=F_{s_1}F_{s_2}\dots F_{s_n}			
			\]
			for $s_i\in S$. As a graded object, this is the direct sum of Bott-Samelson 
			objects indexed by the $2^n$ subexpressions of $\uw=s_1s_2\dots s_n$ (i.e., the sequences in $\{0,1\}^n$).
			We define a new complex $R^\bullet_{\uw}$ which, 
			as a graded object, is a direct sum indexed over \emph{subwords} (each of which can correspond
			to several subexpressions). More precisely, for each subword $\ux$ of $\uw$, let $C_{\ux}$ be the Bott-Samelson object 
			$B_{\ux^*}(\ell(\ux^*)-\ell(\ux))$, where $\ux^*$ is the word obtained from $\ux$ by contracting
			all sequences of repeated letters to a single occurrence (for example, if $\ux=ssttss$, then $\ux^*=sts$).
			The graded piece $R^q_{\uw}$ is the sum, over all the 
			subwords $\ux$ of $\uw$ such that $\ell(\uw)-\ell(\ux)=q$, 
			of the $C_{\ux}$'s. The differential is defined in a natural way: 
			roughly speaking, there is an arrow from $C_{\ux}$ to $C_{\uz}$ when $\uz$ is a subword of $\ux$ 
			(see \S \ref{subs_reducedcom} for more detail). The complex thus obtained is a summand of the original complex $\Roudg{\uw}$, and 
			we show the following (see Theorem \ref{thm_reduce} below).
			\begin{thm*}
				The inclusion and projection maps of the complex 
				$R^\bullet_{\uw}$ in $\Roudg{\uw}$ are mutually inverse
				homotopy equivalences.
			\end{thm*}
			\begin{exa}
				Let $\uw=ssttss$. The complex $\Roudg{\uw}$ is drawn in Figure \ref{fig_6cube}. 
				It is a cube of dimension $6$, whose vertices are labelled by the $2^6=64$ subexpressions of 
				$ssttss$. The arrows (the edges of the cube) describe the differential map: each subexpression
				has a map towards those subexpressions that can be obtained from it by turning a 1 to a 0.
				\begin{figure}[h]
					\centering
					\begin{tikzpicture}[x=1.7cm,y=.15cm]
						\coordinate (v000000) at (0.3,0);
						\coordinate (v000001) at (1,5);\coordinate (v000010) at (1,3);\coordinate (v000100) at (1,1);\coordinate (v001000) at (1,-1);\coordinate (v010000) at (1,-3);\coordinate (v100000) at (1,-5);
						\coordinate (v000011) at (2,14);\coordinate (v000101) at (2,12);\coordinate (v000110) at (2,10);\coordinate (v001001) at (2,8);\coordinate (v001010) at (2,6);\coordinate (v001100) at (2,4);\coordinate (v010001) at (2,2);\coordinate (v010010) at (2,0);\coordinate (v010100) at (2,-2);\coordinate (v011000) at (2,-4);\coordinate (v100001) at (2,-6);\coordinate (v100010) at (2,-8);\coordinate (v100100) at (2,-10);\coordinate (v101000) at (2,-12);\coordinate (v110000) at (2,-14);
						\coordinate (v000111) at (3,19);\coordinate (v001011) at (3,17);\coordinate (v001101) at (3,15);\coordinate (v001110) at (3,13);\coordinate (v010011) at (3,11);\coordinate (v010101) at (3,9);\coordinate (v010110) at (3,7);\coordinate (v011001) at (3,5);\coordinate (v011010) at (3,3);\coordinate (v011100) at (3,1);\coordinate (v100011) at (3,-1);\coordinate (v100101) at (3,-3);\coordinate (v100110) at (3,-5);\coordinate (v101001) at (3,-7);\coordinate (v101010) at (3,-9);\coordinate (v101100) at (3,-11);\coordinate (v110001) at (3,-13);\coordinate (v110010) at (3,-15);\coordinate (v110100) at (3,-17);\coordinate (v111000) at (3,-19);
						\coordinate (v001111) at (4,14);\coordinate (v010111) at (4,12);\coordinate (v011011) at (4,10);\coordinate (v011101) at (4,8);\coordinate (v011110) at (4,6);\coordinate (v100111) at (4,4);\coordinate (v101011) at (4,2);\coordinate (v101101) at (4,0);\coordinate (v101110) at (4,-2);\coordinate (v110011) at (4,-4);\coordinate (v110101) at (4,-6);\coordinate (v110110) at (4,-8);\coordinate (v111001) at (4,-10);\coordinate (v111010) at (4,-12);\coordinate (v111100) at (4,-14);
						\coordinate (v011111) at (5,5);\coordinate (v101111) at (5,3);\coordinate (v110111) at (5,1);\coordinate (v111011) at (5,-1);\coordinate (v111101) at (5,-3);\coordinate (v111110) at (5,-5);
						\coordinate (v111111) at (5.7,0);
						\foreach \a in {0,1}{
							\foreach \b in {0,1}{
								\foreach \c in {0,1}{
									\foreach \d in {0,1}{
										\foreach \e in {0,1}{
											\foreach \f in {0,1}{
													\fill (v\a\b\c\d\e\f) circle (1pt);												
													\node (n\a\b\c\d\e\f) at (v\a\b\c\d\e\f) {};
												}
											}
										}
									}
								}
							}
						\foreach \b in {0,1}{
							\foreach \c in {0,1}{
								\foreach \d in {0,1}{
									\foreach \e in {0,1}{
										\foreach \f in {0,1}{
												\draw[-latex,gray] (n0\b\c\d\e\f)--(n1\b\c\d\e\f);
											}
										}
									}
								}
							}
						\foreach \a in {0,1}{
								\foreach \c in {0,1}{
									\foreach \d in {0,1}{
										\foreach \e in {0,1}{
											\foreach \f in {0,1}{
												\draw[-latex,gray] (n\a0\c\d\e\f)--(n\a1\c\d\e\f);												
												}
											}
										}
									}
							}
						\foreach \a in {0,1}{
							\foreach \b in {0,1}{
									\foreach \d in {0,1}{
										\foreach \e in {0,1}{
											\foreach \f in {0,1}{
												\draw[-latex,gray] (n\a\b0\d\e\f)--(n\a\b1\d\e\f);
												}
											}
										}
								}
							}
						\foreach \a in {0,1}{
							\foreach \b in {0,1}{
								\foreach \c in {0,1}{
										\foreach \e in {0,1}{
											\foreach \f in {0,1}{
												\draw[-latex,gray] (n\a\b\c0\e\f)--(n\a\b\c1\e\f);
												}
											}
										}
									}
								}
						\foreach \a in {0,1}{
							\foreach \b in {0,1}{
								\foreach \c in {0,1}{
									\foreach \d in {0,1}{
											\foreach \f in {0,1}{
												\draw[-latex,gray] (n\a\b\c\d0\f)--(n\a\b\c\d1\f);
												}
											}
										}
									}
								}
						\foreach \a in {0,1}{
							\foreach \b in {0,1}{
								\foreach \c in {0,1}{
									\foreach \d in {0,1}{
										\foreach \e in {0,1}{
												\draw[-latex,gray] (n\a\b\c\d\e0)--(n\a\b\c\d\e1);
												}
											}
										}
									}
								}
					\end{tikzpicture}
					\caption{A picture of the complex $\Roudg{\underline{ssttss}}$.}\label{fig_6cube}
				\end{figure}
				The complex $R^\bullet_{\uw}$ is showed in Figure \ref{fig_cpxssttss}. This time the vertices are labelled by the 
				23 subwords of $\uw$, and the arrows 
				(representing the new differential) connect each word with all its subwords 
				obtained by eliminating a letter (the symbol $\emptyset$ denotes the empty word).
				\begin{figure}[h]
					\centering
					\begin{tikzpicture}[x=1.5cm,y=.45cm]
						\coordinate (v) at (5.6,0);
						\coordinate (s) at (5,1);\coordinate (t) at (5,-1);
						\coordinate (st) at (4,3);\coordinate (ss) at (4,1);\coordinate (tt) at (4,-1);\coordinate (ts) at (4,-3);
						\coordinate (sst) at (3,5);\coordinate (stt) at (3,3);\coordinate (sss) at (3,1);\coordinate (sts) at (3,-1);\coordinate (tss) at (3,-3);\coordinate (tts) at (3,-5);
						\coordinate (sstt) at (2,5);\coordinate (ssts) at (2,3);\coordinate (ssss) at (2,1);\coordinate (stts) at (2,-1);\coordinate (stss) at (2,-3);\coordinate (ttss) at (2,-5);
						\coordinate (sstss) at (1,0);\coordinate (sstts) at (1,3);\coordinate (sttss) at (1,-3);
						\coordinate (ssttss) at (0,0);
						\node (nv) at (v) {$C_{\emptyset}$};
						\foreach \i in {s,t,ss,st,ts,tt,sss,sst,sts,stt,tss,tts,ssss,ssts,sstt,stss,stts,ttss,sstts,sstss,sttss,ssttss}{
							\node (n\i) at (\i) {$C_{\i}$};
						}
						\draw[-latex,gray] (nssttss)--(nsstts);\draw[-latex,gray] (nssttss)--(nsstss);\draw[-latex,gray] (nssttss)--(nsttss);
						\draw[-latex,gray] (nsstts)--(nsstt);\draw[-latex,gray] (nsstts)--(nssts);\draw[-latex,gray] (nsstts)--(nstts);
						\draw[-latex,gray] (nsstss)--(nssts);\draw[-latex,gray] (nsstss)--(nssss);\draw[-latex,gray] (nsstss)--(nstss);
						\draw[-latex,gray] (nsttss)--(nstts);\draw[-latex,gray] (nsttss)--(nstss);\draw[-latex,gray] (nsttss)--(nttss);
						\draw[-latex,gray] (nssss)--(nsss);
						\draw[-latex,gray] (nssts)--(nsst);\draw[-latex,gray] (nssts)--(nsss);\draw[-latex,gray] (nssts)--(nsts);
						\draw[-latex,gray] (nsstt)--(nsst);\draw[-latex,gray] (nsstt)--(nstt);
						\draw[-latex,gray] (nstss)--(nsts);\draw[-latex,gray] (nstss)--(nsss);\draw[-latex,gray] (nstss)--(ntss);
						\draw[-latex,gray] (nstts)--(nstt);\draw[-latex,gray] (nstts)--(nsts);\draw[-latex,gray] (nstts)--(ntts);
						\draw[-latex,gray] (nttss)--(ntts);\draw[-latex,gray] (nttss)--(ntss);
						\draw[-latex,gray] (nsss)--(nss);
						\draw[-latex,gray] (nsst)--(nss);\draw[-latex,gray] (nsst)--(nst);
						\draw[-latex,gray] (nsts)--(nst);\draw[-latex,gray] (nsts)--(nss);\draw[-latex,gray] (nsts)--(nts);
						\draw[-latex,gray] (nstt)--(nst);\draw[-latex,gray] (nstt)--(ntt);
						\draw[-latex,gray] (ntss)--(nts);\draw[-latex,gray] (ntss)--(nss);
						\draw[-latex,gray] (ntts)--(ntt);\draw[-latex,gray] (ntts)--(nts);
						\draw[-latex,gray] (nss)--(ns);
						\draw[-latex,gray] (nst)--(ns);\draw[-latex,gray] (nst)--(nt);
						\draw[-latex,gray] (nts)--(nt);\draw[-latex,gray] (nts)--(ns);
						\draw[-latex,gray] (ntt)--(nt);
						\draw[-latex,gray] (ns)--(nv);\draw[-latex,gray] (nt)--(nv);
					\end{tikzpicture}
					\caption{The complex $F_{\underline{ssttss}}$.}\label{fig_cpxssttss}
				\end{figure}
			\end{exa}
			This reduction is obtained via 
			\emph{Morse theoretical Gaussian elimination for complexes}.
			This is an adaptation to the case of complexes over additive categories of a work by 
			Sk\"oldberg \cite{Sko}, inspired by discrete Morse theory in the sense of Forman \cite{For}.
			See \S \ref{sec_morsegauss} for more details.
			\subsection*{Acknowledgments}
			This work is part of the author's PhD thesis, so first of all I would like to thank my advisors Daniel Juteau and Geordie Williamson
			for the support and all the useful discussions. I am also very grateful to Anthony Licata for pointing me 
			to discrete Morse theory. Finally I would like to thank the 
			referees for their careful reading of the manuscript that resulted in a substantial improvement of the paper.
		\section{The diagrammatic Hecke category}\label{sec_diagHeck}
		We now recall the construction, 
		by Elias and Williamson \cite{EW}, of the \emph{diagrammatic Hecke category} associated with 
		(a realization of) an arbitrary Coxeter system.	We begin with some background about Coxeter
		systems, braid groups and Hecke algebras.
		\subsection{Notation for Coxeter systems and braid groups}
		Let $(W,S)$ be a Coxeter system. For each $s,t\in S$ let $m_{st}$ be the order of $st$ in $W$ (in particular $m_{ss}=1$),
		so that
		\[
			W=\langle s\in S\mid (st)^{m_{st}}=1, \text{ if } m_{st}<\infty \rangle.
		\]
		We will restrict to the case $|S|<\infty$ of finitely generated Coxeter systems. 

		Let $\cword$ be the free monoid generated by $S$. Its elements are called \emph{Coxeter words}
		and will be denoted by underlined letters. 
		If $\uw$ is a Coxeter word, we say that it \emph{expresses}
		the element $w\in W$ if $w$ is its image via the natural morphism 
		$\cword\rightarrow W$. 
		We say that a word is \emph{monotonous} if it contains only (repetitions of) one letter.		
		Let $\ell(\uw)$ denote the length of $\uw$, i.e.\ the number of its letters.
		The word $\uw$ is said to be \emph{reduced} if there is no shorter word expressing $w$.

		For $\uw, \ux \in \cword$, we say that $\ux$ is a \emph{subword} of $\uw$, and we 
		write $\ux\preceq \uw$, 
		if $\ux$ is obtained from $\uw$ by erasing some letters. In other words,
		if $\uw=s_1s_2\cdots s_n$ with $s_i\in S$, then 
		$\ux$ is of the form $s_{i_1}s_{i_2}\cdots s_{i_r}$, 
		with $1\le i_1 < i_2 < \dots < i_r\le n$. A \emph{subexpression} is instead the datum of a subword
		together with the information of the precise positions of the letters in the original word. For a given $\uw\in \cword$, 
		this can be encoded in a \emph{01-sequence} $\boi\in\{0,1\}^n$ where the ones correspond to the letters
		that we pick and the zeros to those we do not pick. Among all 01-sequences corresponding to a given subword $\ux\preceq\uw$,
		we call \emph{primary} the least one in the lexicographic order (induced by 0<1).
		\begin{exa}
			Let $\uw=stussu$. Then $\ux=sts$ is a subword, obtained by erasing the two $u$'s and one of the two
			$s$'s between them. It corresponds to both subexpressions $110010$ and $110100$. The former is the primary one.
		\end{exa}		
		
%
		

		Recall that to a Coxeter system $(W,S)$ we associate the
		\emph{Artin-Tits group} (or \emph{generalized braid group}):
		\[
			B_W=\langle \sigma_s, s\in S\mid \underbrace{\sigma_s\sigma_t \dots}_{m_{st}}=\underbrace{\sigma_t\sigma_s \dots}_{m_{st}} \rangle.
		\]
		Let $\Sigma^+=\{\sigma_s\}_{s\in S}$,
		$\Sigma^-=\{\sigma_s^{-1}\}_{s\in S}$ and $\Sigma=\Sigma^+ \sqcup \Sigma^-$. 
		Let $\bword$ be the free monoid generated by $\Sigma$. 
		Its elements are called 
		\emph{braid words}. As above, if $\uom$ is a braid word, 
		we say that it \emph{expresses} the element $\omega\in B_W$ 
		if $\uom$ is its image via the natural morphism $\bword\rightarrow B_W$.
		A braid $\omega\in B_W$ is called \emph{positive} if it belongs to the monoid generated by $\Sigma^+$, namely
		it can be expressed by a braid word with only letters from $\Sigma^+$.		
		\subsection{The Hecke algebra}\label{subs_Heckealg}
		Consider the ring $\Laur$ of Laurent polynomial in one 
		variable with integer coefficients.
		\begin{dfn}
			The \emph{Hecke algebra} $\HH$ associated with $(W,S)$ is the 
			$\Laur$-algebra generated by $\{H_s\}_{s\in S}$
			 with relations 
			\[
				\begin{cases}
					(H_s+v)(H_s-v^{-1})=0& 	\text{for all $s\in S$}\\
					\underbrace{H_sH_t\dots}_{m_{st}}=\underbrace{H_tH_s\dots}_{m_{st}}& 		\text{if $m_{st}<\infty$}
				\end{cases}
			\]			
		\end{dfn}		 		
		Notice that $\HH$ can be seen as the quotient of the group algebra of the braid group
		$B_W$ over $\Laur$, by the above $v$-deformed version of the 
		involution relation in $W$. 
		
		For each $s\in S$ let $b_s:=H_s+v$. By the first relation, one gets
		\begin{equation}\label{eq_bsbsalgebra}
			b_sb_s=vb_s+v^{-1}b_s
		\end{equation}
		The category that we are going to describe is a \emph{categorification} 
		of the Hecke algebra. We need to introduce the notion
		of \emph{realization} of a Coxeter system.
		\subsection{Realizations of Coxeter systems}\label{sec_realCox} Let $(W,S)$ be a Coxeter system
		and $\kk$ a commutative ring. A \emph{realization}
		of $W$, in the sense of \cite[\S 3.1]{EW}, is a free,
		finite rank, $\kk$-module $\mathfrak{h}$, with distinguished elements 
		$\{\alpha_s^\vee\}_{s\in S} \subset \mathfrak{h}$ and 
		$\{ \alpha_s \}_{s\in S} \subset \mathfrak{h}^*=\Hom(\mathfrak{h},k)$,
		that we call respectively \emph{simple coroots} and \emph{simple roots}, 
		satisfying the following conditions.
		\begin{enumerate}
		\item If $\langle \cdot\,,\cdot\rangle:\mathfrak{h}^*\times \mathfrak{h}\rightarrow \kk$ is the natural
			pairing, then $\langle \alpha_s,\alpha_s^\vee\rangle =2$, for each $s\in S$.
		\item\label{item_cond2} The $\kk$-module $\mathfrak{h}$ is a representation of $W$ via 
			\begin{equation}
				s(v)=v-\langle \alpha_s, v\rangle \alpha_s^{\vee},\quad\forall v\in \mathfrak{h},\,\forall s \in S.		
			\end{equation}
			Then $W$ acts on $\mathfrak{h}^*$ by the contragredient representation, described by similar formulas:
			\begin{equation}\label{eq_realiz2}
				s(\lambda)=\lambda-\langle \lambda, \alpha_s^\vee \rangle \alpha_s,\quad\quad\forall \lambda\in \mathfrak{h}^*,\,\forall s \in S.		
			\end{equation}		
		\item The technical condition \cite[\S 3.1, (3.3)]{EW} is satisfied.
		\end{enumerate}
%
		The action of the realization over simple roots and coroots is encoded in the coefficients\footnote{For 
		the notation, we follow Bourbaki \cite{Bour}: notice that the pairing considered here is the transposed
		of that of \cite{EW}, and \cite{Kac}.}
		$a_{st}:=\langle \alpha_s,\alpha_t^\vee\rangle$, that we store in the 
		\emph{Cartan matrix} $( a_{st} )_{s,t\in S}$.

		Let $R=S(\mathfrak{h}^*)$, with $\mathfrak{h}^*$ in degree 2. The action of $W$ on 
		$\mathfrak{h}^*$ defined by \eqref{eq_realiz2} extends naturally to $R$. 
		We define, for each $s\in S$, the \emph{Demazure operator} $\partial_s : R \rightarrow R$, via:
		\[
						f \mapsto \frac{f-s(f)}{\alpha_s}.
		\]
		Then, in particular, we have\footnote{%
			Notice that $\partial_s(\alpha_t)$ is the same as in \cite{EW}, so this notation avoids 
			transposition issues.}
		$\partial_s(\alpha_t)=\langle \alpha_t,\alpha_s^\vee\rangle$. Notice also that when $f$ is $s$-invariant we have $\partial_s(f)=0$.
		
		In this paper we will only consider \emph{balanced} realizations (see \cite[Definition 3.7]{EW}).
		We also assume \emph{Demazure surjectivity} (see \cite[Assumption 3.9]{EW}), namely we suppose 
		that, for each $s\in S$, the maps:
		\begin{align*}
			&\langle \alpha_s,\cdot \rangle: \mathfrak{h}\rightarrow \kk&
			&\text{and}&
			&\langle \cdot\,,\alpha_s^\vee \rangle: \mathfrak{h}^*\rightarrow \kk
		\end{align*}			
		are surjective. In this case we can (and do) choose some $\delta_s\in\mathfrak{h}^*$ such that 
		$\langle \delta_s,\alpha_s^\vee\rangle=1$.
		Notice that if $2$ is invertible in $\kk$ then this assumption always holds (take for instance
		$\delta_s=\alpha_s/2$). 
		\begin{exa}\label{exa_realiz}
				For any Coxeter group $(W,S)$ one can consider the \emph{geometric representation} 
					$\hh:=\bigoplus_{s\in S}\mathbb{R}\alpha_s^\vee$ with 
					$\alpha_s\in \hh^*$ defined via 
					\[
						\langle \alpha_s,\alpha_t^\vee\rangle=-2\cos(\pi/m_{st}),
					\]
					with $m_{ss}:=1$ as usual, and $\pi/\infty:=0$. 
					Then the simple roots are linearly independent if and only if $W$ is finite.
		\end{exa}
		For more details and examples, see \cite[\S 3.1]{EW}.
		\subsection{Definition of the category}\label{subs_defcatHeck}
		Given a realization $\mathfrak{h}$ over $\kk$ of a Coxeter system $(W,S)$, 
		one constructs the corresponding diagrammatic Hecke category $\D=\D(\mathfrak{h},\kk)$.
		This is a $\kk$-linear monoidal category enriched in graded $R$-bimodules.
		First one defines the Bott-Samelson category $\D_{\mathrm{BS}}$ by generators and relations,
		then one gets $\D$ as the Karoubi envelope of the closure of $\D_{\mathrm{BS}}$ by direct 
		sums and shifts.	
		\begin{enumerate}
			\item The objects of $\DBS$ are generated by tensor product from objects $B_s$ for $s\in S$. So a general object
				corresponds to a Coxeter word: if
				$\uw=s_1\dots s_n$, let $B_{\uw}$ denote the object
				$B_{s_1}\otimes \cdots \otimes B_{s_n}$. Let also $\un$ denote the monoidal unit, corresponding to the 
				empty word. 
			\item Morphisms in $\Hom_{\DBS}(B_{\uw_1},B_{\uw_2})$
				are $\kk$-linear combinations of \emph{Soergel graphs}, which are defined as follows.
				\begin{itemize}
					\item We associate a color to each simple reflection.
					\item A Soergel graph is then a colored, \emph{decorated} planar graph 
						contained in the planar strip $\mathbb{R}\times [0,1]$, with boundary in 
						$\mathbb{R}\times\{0,1\}$.
					\item The bottom (and top) boundary is the arrangement
						of boundary points colored according to the letters of the source word $\uw_1$ 
						(and target word $\uw_2$, respectively).
					\item The edges of the graph are colored in such a way that those connected with the 
						boundary have colors consistent with the boundary points.
					\item The other vertices of the graph are either: 
						\begin{enumerate}
							\item univalent (called \emph{dots}), which are declared of degree 1, or;
							\item trivalent	with three edges of the same color, of degree $-1$, or;
							\item $2m_{st}$-valent with edges of alternating colors corresponding to $s$ and $t$,
							 	if $m_{st}<\infty$, of degree 0.
								By poetic licence, we also call them $(s,t)$-ars.
						\end{enumerate}
						
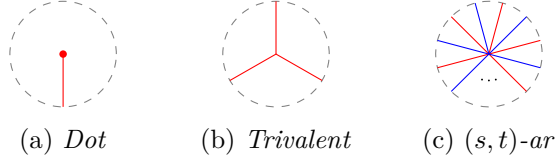
\begin{center}
							\begin{tikzpicture}[baseline=0,scale=.7]
								\def\shift{4cm}
								\draw[gray,dashed] (0,0) circle (1cm); %
									\draw[red] (0,-1)--(0,0); \fill[red] (0,0) circle (2pt);
								\node[anchor=north,inner sep=8pt] at (0,-1) {(a) \textit{Dot}};
								\begin{scope}[xshift=\shift]
									\draw[gray,dashed] (0,0) circle (1cm); %
										\draw[red] (0,1)--(0,0); %
											\draw[red] (-30:1cm)--(0,0)--(-150:1cm);
									\node[anchor=north,inner sep=8pt] at (0,-1) {(b) \textit{Trivalent}};
								\end{scope}
								\begin{scope}[xshift=2*\shift,rotate=-15]
									\draw[gray,dashed] (0,0) circle (1cm); %
										\draw[red] (0,1)--(0,0); %
											\draw[red] (-30:1cm)--(150:1cm);%
												\draw[red](-150:1cm)--(30:1cm);
										\draw[blue,rotate=30] (0,1)--(0,0); %
											\draw[blue,rotate=30] (-30:1cm)--(150:1cm);%
												\draw[blue,rotate=30](-150:1cm)--(30:1cm);
									\foreach \i in {-60,-75,-90}{%
											\fill (\i:.5cm) circle (.5pt);
										}
								\end{scope}
								\node[anchor=north,xshift=1.4*\shift,inner sep=8pt] at (0,-1) {(c) \textit{$(s,t)$-ar}};
							\end{tikzpicture}
						\end{center}
					\item Decorations are boxes labelled by homogeneous elements in 
						$R$ that can appear in
						any region (i.e.\ connected component of the complement of the graph): 
						we will usually omit the boxes and just write the polynomials.
				\end{itemize}
				Then, composition of morphisms is given by gluing diagrams vertically, whereas tensor product
				is given by gluing them horizontally. The identity morphism of the object $B_{\uw}$ is
				the diagram with parallel vertical strands colored according to the word $\uw$.
			\item These diagrams are identified via some relations:
				\begin{description}
					\item[Polynomial relations.] 
%
						The first relations impose additivity and multiplicativity of polynomial boxes, more precisely:
						\begin{align}
							& \text{\emph{Addition}:}  &
							\begin{tikzpicture}[baseline=-0.1cm,scale=.7, transform shape]
								\draw[gray,dashed] (0,0) circle (1cm); \node[draw] at (0,0) {$f$};
							\end{tikzpicture}
							+
							\begin{tikzpicture}[baseline=-0.1cm,scale=.7, transform shape]
								\draw[gray,dashed] (0,0) circle (1cm); \node[draw] at (0,0) {$g$};
							\end{tikzpicture}
							&=
							\begin{tikzpicture}[baseline=-0.1cm,scale=.7, transform shape]
								\draw[gray,dashed] (0,0) circle (1cm); \node[draw] at (0,0) {$f+g$};
							\end{tikzpicture}
							\\
							& \text{\emph{Multiplication}:}  &
							\begin{tikzpicture}[baseline=-0.1cm,scale=.7, transform shape]
								\draw[gray,dashed] (0,0) circle (1cm); \node[draw] at (-0.4,0.4) {$f$};\node[draw] at (0.4,-0.4) {$g$};
							\end{tikzpicture}
							&=
							\begin{tikzpicture}[baseline=-0.1cm,scale=.7, transform shape]
								\draw[gray,dashed] (0,0) circle (1cm); \node[draw] at (0,0) {$fg$};
							\end{tikzpicture}
						\end{align}					
						This gives morphism spaces the structure of $R$-bimodules, by acting on the leftmost or the 
						rightmost region. 
						
						The other polynomial relations are:
						\begin{align}
							& \text{\emph{Barbell} relation:}  &
							\begin{tikzpicture}[baseline=-0.1cm,scale=.7, transform shape] 
								\draw[gray,dashed] (0,0) circle (1cm); \draw[red] (0,.5)--(0,-.5); \fill[red] (0,.5) circle (2pt);\fill[red] (0,-.5) circle (2pt);
							\end{tikzpicture}
							&=
							\begin{tikzpicture}[baseline=-0.1cm,scale=.7, transform shape]
								\draw[gray,dashed] (0,0) circle (1cm); \node[draw] at (0,0) {$\alpha_{\re{s}}$};
							\end{tikzpicture}			\\
							& \text{\emph{Sliding} relation:} & 
							\begin{tikzpicture}[baseline=-0.1cm,scale=.7, transform shape]
								\draw[gray,dashed] (0,0) circle (1cm); \draw[red] (0,1)--(0,-1); 
								\node[draw] at (.5,0) {$f$};
							\end{tikzpicture}			
							&=
							\begin{tikzpicture}[baseline=-0.1cm,scale=.7, transform shape]
								\draw[gray,dashed] (0,0) circle (1cm); \draw[red] (0,1)--(0,-1); 
								\node[draw, inner sep =.08cm] at (-.5,0) {$\re{s}(f)$};
							\end{tikzpicture}			
							+
							\begin{tikzpicture}[baseline=-0.1cm,scale=.7, transform shape]
								\draw[gray,dashed] (0,0) circle (1cm); \draw[red] (0,1)--(0,.5); \fill[red] (0,.5) circle (2pt);
								\draw[red] (0,-1)--(0,-.5); \fill[red] (0,-.5) circle (2pt); 
								\node[draw] at (0,0) {$\partial_{\re{s}}(f)$};
							\end{tikzpicture}\label{eq_relsliding}
							\end{align}
							In particular Relation \eqref{eq_relsliding} implies that
							an $s$-invariant $f$ can slide across strands of the color of $s$ without other changes.
					\item[One color relations.] These are the following:
						\begin{align}
							& \text{\emph{Frobenius associativity}:} &
							\begin{tikzpicture}[baseline=-0.1cm,color=red,scale=sqrt(2)/3,scale=.7, transform shape]
								\draw (-1.5,-1.5)--(-0.6,0)--(0.6,0)--(1.5,-1.5);
								\draw (-1.5,1.5)--(-0.6,0); \draw(0.6,0)--(1.5,1.5);
								\draw[dashed,gray] (0,0) circle ({3/sqrt(2)});
							\end{tikzpicture}
							&=
							\begin{tikzpicture}[baseline=-0.1cm,color=red,scale=sqrt(2)/3,rotate=90,scale=.7, transform shape]
								\draw (-1.5,-1.5)--(-0.6,0)--(0.6,0)--(1.5,-1.5);
								\draw (-1.5,1.5)--(-0.6,0); \draw(0.6,0)--(1.5,1.5);
								\draw[dashed,gray] (0,0) circle ({3/sqrt(2)});
							\end{tikzpicture} \label{eq_frobass} \\
							& \text{\emph{Frobenius unit}:} &
							\begin{tikzpicture}[color=red,baseline=-0.1cm,scale=.7, transform shape]
								\draw[gray,dashed] (0,0) circle (1cm);
								\draw (0,-1) -- (0,1); \draw (0,0)--(.3,0);\fill (0.3,0) circle (2pt);
							\end{tikzpicture}
							&=
							\begin{tikzpicture}[baseline=-0.1cm,scale=.7, transform shape]
								\draw[gray,dashed] (0,0) circle (1cm);						
								\draw[red] (0,-1) -- (0,1);
							\end{tikzpicture} \label{eq_frobunit}
							\\
							&\text{\emph{Needle} relation:} &
							\begin{tikzpicture}[baseline=0,scale=.7, transform shape]
								\draw[gray,dashed] (0,0) circle (1cm);
								\draw[red] (0,-1)--(0,-0.3) arc (-90:270:0.3cm);
							\end{tikzpicture}
							&= 0 \label{eq_needle}
						\end{align}
						\begin{rmk}\label{rmk_selfbiadj}				
							Relations \eqref{eq_frobass} and \eqref{eq_frobunit} could be phrased by saying that
							$B_s$ is a \emph{Frobenius algebra object} in the category $\DBS$, which explains the terminology. 
							This was pointed out in 
							\cite{EliKho}. See \cite[\S 8]{EMTW}.
						\end{rmk}
					\item[Two color relations.] These allow to move dots,
						or trivalent vertices, across $(s,t)$-ars.
						We give the two versions, according to the parity of $m_{st}$:
						\[
							\begin{tikzpicture}[baseline=0,scale=.8,transform shape]
								\draw[dashed,gray] (0,0) circle (1cm);
								\coordinate (a) at (-.3,0); \coordinate (b) at (0.5,0);
								\draw[red] (180:1cm)--(a); \draw[blue] ({180-360/11}:1cm)--(a);
								\draw[red] ({180-2*360/11}:1cm)--(a); \draw[blue] ({180-5*360/11}:1cm)--(b);
								\draw[red] ({180-4*360/11}:1cm)--(a);
								\draw[blue] ({180+360/11}:1cm)--(a);
								\draw[red] ({180+2*360/11}:1cm)--(a); \draw[blue] ({180+5*360/11}:1cm)--(b);
								\draw[red] ({180+4*360/11}:1cm)--(a);
								\draw[blue](a)--(b);
								\node at (-0.05,.5) {...};\node at (-0.05,-.5) {...};
							\end{tikzpicture}
							=
							\begin{tikzpicture}[baseline=0,scale=.8,transform shape]
								\draw[dashed,gray] (0,0) circle (1cm);
								\coordinate (a) at (-.5,0); \coordinate (b) at (0.15,0.35); \coordinate (c) at (0.15,-0.35);
								\draw[red] (180:1cm)--(a); \draw[blue] ({180-360/11}:1cm)--(b);
								\draw[red] ({180-2*360/11}:1cm)--(b); \draw[blue] ({180-5*360/11}:1cm)--(b);
								\draw[red] ({180-4*360/11}:1cm)--(b);
								\draw[blue] ({180+360/11}:1cm)--(c);
								\draw[red] ({180+2*360/11}:1cm)--(c); \draw[blue] ({180+5*360/11}:1cm)--(c);
								\draw[red] ({180+4*360/11}:1cm)--(c);
								\draw[blue](c)..controls (-0.15,0)..(b); \draw[red] (b) -- (c);\draw[red] (b)--(a)--(c);
								\draw[red] (b) ..controls (0.7,0).. (c);
								\node at (.2,.6) {...};\node at (.2,-.6) {...};\node at (.35,0) {...};
							\end{tikzpicture}\quad\text{or}\quad
							\begin{tikzpicture}[baseline=0,scale=.8,transform shape]
								\draw[dashed,gray] (0,0) circle (1cm);
								\coordinate (a) at (-.3,0); \coordinate (b) at (0.5,0);
								\draw[red] (180:1cm)--(a); \draw[blue] ({180-360/11}:1cm)--(a);
								\draw[red] ({180-2*360/11}:1cm)--(a); \draw[red] ({180-5*360/11}:1cm)--(b);
								\draw[blue] ({180-4*360/11}:1cm)--(a);
								\draw[blue] ({180+360/11}:1cm)--(a);
								\draw[red] ({180+2*360/11}:1cm)--(a); \draw[red] ({180+5*360/11}:1cm)--(b);
								\draw[blue] ({180+4*360/11}:1cm)--(a);
								\draw[red](a)--(b);
								\node at (-0.05,.5) {...};\node at (-0.05,-.5) {...};
							\end{tikzpicture}
							=
							\begin{tikzpicture}[baseline=0,scale=.8,transform shape]
								\draw[dashed,gray] (0,0) circle (1cm);
								\coordinate (a) at (-.5,0); \coordinate (b) at (0.15,0.35); \coordinate (c) at (0.15,-0.35);
								\draw[red] (180:1cm)--(a); \draw[blue] ({180-360/11}:1cm)--(b);
								\draw[red] ({180-2*360/11}:1cm)--(b); \draw[red] ({180-5*360/11}:1cm)--(b);
								\draw[blue] ({180-4*360/11}:1cm)--(b);
								\draw[blue] ({180+360/11}:1cm)--(c);
								\draw[red] ({180+2*360/11}:1cm)--(c); \draw[red] ({180+5*360/11}:1cm)--(c);
								\draw[blue] ({180+4*360/11}:1cm)--(c);
								\draw[blue](c)..controls (-0.15,0)..(b); \draw[red] (b) -- (c);\draw[red] (b)--(a)--(c);
								\draw[blue] (b) ..controls (0.7,0).. (c);
								\node at (.2,.6) {...};\node at (.2,-.6) {...};\node at (.35,0) {...};
							\end{tikzpicture}
						\]
						and
						\[
							\begin{tikzpicture}[baseline=0,scale=.8,transform shape]
								\draw[dashed,gray] (0,0) circle (1cm);
								\coordinate (a) at (-.2,0); \coordinate (b) at (0.5,0);
								\draw[red] (180:1cm)--(a); \draw[blue] ({180-360/11}:1cm)--(a);
								\draw[blue] ({180-3*360/11}:1cm)--(a); 
								\draw[red] ({180-4*360/11}:1cm)--(a);
								\draw[blue] ({180+360/11}:1cm)--(a);
								\draw[blue] ({180+3*360/11}:1cm)--(a); \fill[blue] (b) circle (2pt);
								\draw[red] ({180+4*360/11}:1cm)--(a);
								\draw[blue](a)--(b);
								\node at (-0.4,.5) {...};\node at (-0.4,-.5) {...};
							\end{tikzpicture}
							=
							\begin{tikzpicture}[baseline=0,scale=.8,transform shape]
								\draw[dashed,gray] (0,0) circle (1cm);
								\coordinate (a) at (-.2,0); \coordinate (b) at (0.5,0);
								\draw[red] (180:1cm)--(a); \draw[blue] ({180-360/11}:1cm)--(a);
								\draw[blue] ({180-3*360/11}:1cm)--(a); 
								\draw[blue] ({180+360/11}:1cm)--(a);
								\draw[blue] ({180+3*360/11}:1cm)--(a);
								\draw[red] ({180+4*360/11}:1cm) -- (b) -- ({180-4*360/11}:1cm); \draw[red](a)--(b);
								\node at (-0.4,.6) {...};\node at (-0.4,-.6) {...};
								\node[draw,circle, inner sep=.1cm,fill=white] at (a) {$\JW$};
							\end{tikzpicture}
							\quad\text{or}\quad
							\begin{tikzpicture}[baseline=0,scale=.8,transform shape]
								\draw[dashed,gray] (0,0) circle (1cm);
								\coordinate (a) at (-.2,0); \coordinate (b) at (0.5,0);
								\draw[red] (180:1cm)--(a); \draw[blue] ({180-360/11}:1cm)--(a);
								\draw[red] ({180-3*360/11}:1cm)--(a); 
								\draw[blue] ({180-4*360/11}:1cm)--(a);
								\draw[blue] ({180+360/11}:1cm)--(a);
								\draw[red] ({180+3*360/11}:1cm)--(a); \fill[red] (b) circle (2pt);
								\draw[blue] ({180+4*360/11}:1cm)--(a);
								\draw[red](a)--(b);
								\node at (-0.4,.5) {...};\node at (-0.4,-.5) {...};
							\end{tikzpicture}
							=
							\begin{tikzpicture}[baseline=0,scale=.8,transform shape]
								\draw[dashed,gray] (0,0) circle (1cm);
								\coordinate (a) at (-.2,0); \coordinate (b) at (0.5,0);
								\draw[red] (180:1cm)--(a); \draw[blue] ({180-360/11}:1cm)--(a);
								\draw[red] ({180-3*360/11}:1cm)--(a); 
								\draw[blue] ({180+360/11}:1cm)--(a);
								\draw[red] ({180+3*360/11}:1cm)--(a);
								\draw[blue] ({180+4*360/11}:1cm) -- (b) -- ({180-4*360/11}:1cm); \draw[blue](a)--(b);
								\node at (-0.4,.6) {...};\node at (-0.4,-.6) {...};
								\node[draw,circle, inner sep=.1cm,fill=white] at (a) {$\JW$};
							\end{tikzpicture}
						\]
						where the circles labelled $\JW$ are the \emph{Jones-Wenzel morphisms}. 
						These are certain $\kk$-linear combinations of diagrams (with circular boundary
						and $2m_{st}-2$ boundary points around it)
						that can be described in terms of the 2-colored Temperley-Lieb category. 
						We will not need these relations in this paper. For further details, we refer the reader to 
						\cite[\S 5.2]{EW}, \cite{Elias_two} or \cite[\S 8]{EMTW}.
					\item[Three color relations.] For each finite parabolic subgroup $W_I$
						of rank 3, there is a relation ensuring compatibility between the three 
						corresponding $2m$-valent vertices. This is an analog of the 
						\emph{Zamolodchikov tetrahedron equation} for braided monoidal 
						2-categories, generalized to all types (not just $A$). 
						We will not need them either, so for more details see \cite[\S 1.4.3]{EW}.
%
				\end{description}				
				Notice that all the relations are homogeneous, so the morphism spaces are \emph{graded}
				$R$-bimodules.
		\end{enumerate}
		This completes the definition of $\DBS$.
		\begin{dfn}
			The \emph{diagrammatic Hecke category} $\D$ is the Karoubi envelope of the closure of $\DBS$
			by shifts, denoted by $(\cdot)$, and direct sums.
		\end{dfn}
		\subsection{Main properties of the Hecke category}\label{subs_catofHeck}
		The split Grothendieck group $[\D]_{\oplus}$ is naturally a $\Laur$-algebra: the ring structure is induced by the tensor product
		and the action of $v$ corresponds to the shift (more precisely $v[B]:=[B(1)]$).
		Hence we can state the following result (see \cite[\S 6.6]{EW}).
		\begin{thm}\label{thm_soergcat}
			If $\kk$ is a complete local ring, then $\D$ is Krull-Schmidt and 
			there is a unique isomorphism of $\Laur$-algebras, called \emph{character},
					\[
						\mathrm{ch}:[\D]\overset{\sim}{\rightarrow}\HH
					\]
					sending the class of $B_s$ to $b_s$.
		\end{thm}
		\begin{rmk}\label{rmk_incarn}
			\begin{enumerate}
				\item Under the hypotheses of Theorem \ref{thm_soergcat} one can also classify the
					indecomposable objects: they are parametrized, up to isomorphism and shift, by $W$. 
					Moreover morphism spaces are free as left 
					$R$-modules and one can express their graded rank in terms of $\HH$.
				\item If $\chr(\kk)\neq 2$ and the realization 
					$\mathfrak{h}$ is \emph{reflection faithful} in the sense of \cite[Definition 3.10]{EW},
					then $\D$ is equivalent to the category of Soergel bimodules, see \cite[Theorem\ 6.30]{EW}. 
					This assumption can be dropped by using the variant of the latter category introduced by Abe \cite{Abe}.
				\item\label{item_geom} There are also other incarnations of the Hecke category. For instance, 
					if $W$ is a Weyl group, one can construct a geometric version in terms of parity sheaves over 
					the appropriate (affine) flag variety. See \cite{JMW,RW}.
	\end{enumerate}
		\end{rmk}
		\begin{exa}\label{exa_basicdeco}
			The basic example of a diagrammatic computation in $\D$ is the following (see \cite[\S 5]{EW}). 
			If $s\in S$, then
			\begin{equation}\label{eq_BsBscat}
				B_s\otimes B_s\cong B_s(-1)\oplus B_s(1)			
			\end{equation}
			In fact consider the following maps (where we omit the bottom and top lines):
			\[
			\def\d{1cm}
				\iota_1=	\begin{tikzpicture}[yscale=-1,baseline=-.5*\h]
							\draw[red] (\d,0) -- (1.5*\d,.5*\h) -- (2*\d,0);
							\draw[red] (1.5*\d,.5*\h)--(1.5*\d,\h);
							\node at (1.5*\d,0) {$\delta_s$};
						\end{tikzpicture},
						\qquad
				\pi_1=	\begin{tikzpicture}[baseline=.5*\h]
							\draw[red] (\d,0)--(1.5*\d,.5*\h)--(2*\d,0);
							\draw[red] (1.5*\d,.5*\h)--(1.5*\d,\h);
						\end{tikzpicture},
						\qquad
				\iota_2=	\begin{tikzpicture}[baseline=-.5*\h,yscale=-1]
							\draw[red] (\d,0)--(1.5*\d,.5*\h)--(2*\d,0);
							\draw[red] (1.5*\d,.5*\h)--(1.5*\d,\h);
						\end{tikzpicture},
						\qquad
				\pi_2=	-\begin{tikzpicture}[baseline=.5*\h]
							\draw[red] (\d,0) -- (1.5*\d,.5*\h) -- (2*\d,0);
							\draw[red] (1.5*\d,.5*\h)--(1.5*\d,\h);
							\node[yshift=-1pt] at (1.5*\d,0) {$s(\delta_s)$};
						\end{tikzpicture}.
			\def\d{.5cm}
			\]
			Using the relations, one can check that $\iota_1\pi_1+\iota_2\pi_2=\id_{B_sB_s}$, as well as
			\begin{align*}
				&\pi_1\iota_1=\id_{B_s(-1)},&	 &\pi_2\iota_2=\id_{B_s(1)},&	&\pi_i\iota_j=0\text{ if }i\neq j.
			\end{align*}
			The decomposition \eqref{eq_BsBscat} lifts the equality \eqref{eq_bsbsalgebra} in the Hecke algebra.
		\end{exa}
		For practical reasons, especially in Soergel diagrams, we will use the notation:
		\begin{equation}
			\delta_s':=-s(\delta_s).		
		\end{equation}
		As another example of diagrammatic computation we prove the following generalisation of Relation \eqref{eq_needle}.
		\begin{lem}\label{lem_emptyregion}
			Any empty region (with no decorations) delimited by strands of the same color is zero:
			\begin{equation}
				\begin{tikzpicture}	[baseline=0]
					\draw[dashed,gray] (0,0) circle (1cm);
					\draw[red] (45:.5cm)--(90:.5cm)--(135:.5cm)--(180:.5cm)--(225:.5cm)--(270:.5cm)--(315:.5cm);
					\foreach \i in {45,90,135,180,225,270,315}{
					\draw[red] (\i:.5cm)--(\i:1cm);
					}
					\draw[dashed,red](45:.5cm)--(0:.5cm)--(-45:.5cm);
					\foreach \i in {20,10,0,-10,-20}{
						\fill (\i:.75cm) circle (.5pt);
					}
				\end{tikzpicture}
				=0
			\end{equation}
		\end{lem}
		\begin{proof}
			We proceed by induction on the number $n$ of sides of the region, where a side is delimited by two consecutive trivalent vertices.
			The base case $n=1$ is just Relation \eqref{eq_needle}. 
			Suppose the result true for an $n$-sided region and
			consider a region bounded by $n+1$ red sides. 
			We have:
			\begin{equation}
				\begin{tikzpicture}	[baseline=0]
					\draw[dashed,gray] (0,0) circle (1cm);
					\draw[red] (45:.5cm)--(90:.5cm)--(135:.5cm)--(180:.5cm)--(225:.5cm)--(270:.5cm)--(315:.5cm);
					\foreach \i in {45,90,135,180,225,270,315}{
					\draw[red] (\i:.5cm)--(\i:1cm);
					}
					\draw[dashed,red](45:.5cm)--(0:.5cm)--(-45:.5cm);
					\foreach \i in {20,10,0,-10,-20}{
						\fill (\i:.75cm) circle (.5pt);
					}
					\node[font=\tiny] at (157.5:.6cm) {1};
					\node[font=\tiny] at (112.5:.6cm) {2};
					\node[font=\tiny] at (67.5:.6cm) {3};
					\node[font=\tiny,xshift=-.13cm] at (-157.5:.6cm) {$n\!+\!1$};
					\node[font=\tiny] at (-112.5:.6cm) {$n$};
					\node[font=\tiny,xshift=.05cm] at (-67.5:.7cm) {$n\!-\!1$};					
				\end{tikzpicture}
				=
				\begin{tikzpicture}	[baseline=0]
					\draw[dashed,gray] (0,0) circle (1cm);
					\draw[red] (45:.5cm)--(90:.5cm)--(135:.5cm)--(202.5:.5cm)--(270:.5cm)--(315:.5cm);
					\foreach \i in {45,90,135,270,315}{
					\draw[red] (\i:.5cm)--(\i:1cm);
					}
					\draw[red] (202.5:.5cm)--(202.5:.7cm)--(180:1cm);
					\draw[red] (202.5:.7cm)--(225:1cm);
					\draw[dashed,red](45:.5cm)--(0:.5cm)--(-45:.5cm);
					\foreach \i in {20,10,0,-10,-20}{
						\fill (\i:.75cm) circle (.5pt);
					}
					\node[font=\tiny] at (170:.6cm) {1};
					\node[font=\tiny] at (112.5:.6cm) {2};
					\node[font=\tiny] at (67.5:.6cm) {3};
					\node[font=\tiny] at (-120:.6cm) {$n$};
					\node[font=\tiny,xshift=.05cm] at (-67.5:.7cm) {$n\!-\!1$};					
				\end{tikzpicture}
				=0,
			\end{equation}			
			where, for the first equality, we applied Relation \eqref{eq_frobass} to the two 
			trivalent vertices connected by the strand labelled $n+1$ 
			(the numbers here are only labels, 
			not decorations). For the second equality we applied the induction hypothesis.
		\end{proof}
		\begin{rmk}\label{rmk_onSoergelcat}
			The Hecke category plays an important role in representation theory. 
			We only mention an example related with the origin of Rouquier complexes (see Remark \ref{rmk_orRou} below).
				Let $\mathfrak{g}$ be a complex semisimple Lie algebra and fix a Cartan subalgebra $\mathfrak{h}$. 
					Let $W$ be the associated Weyl group and $S$ the set of simple reflections corresponding to a 
					choice of simple roots. Recall that the \emph{dot action} of $W$ on $\mathfrak{h}^*$ is given 
					by $w\cdot \lambda=w(\lambda+\rho)-\rho$ where $\rho$ is half the sum of positive 
					roots. Consider the category $\mathcal{O}$ of Bernstein-Gel'fand-Gel'fand and its 
					decomposition into blocks $\mathcal{O}=\bigoplus\mathcal{O}_\lambda$. For each 
					$s\in S$ let $\mu\in \mathfrak{h}^*$ be a weight with stabilizer $\{1,s\}$ under 
					the dot action. Then one defines translation functors 
					$T^s:=T_0^{\mu}:\mathcal{O}_0\rightarrow\mathcal{O}_{\mu}$ and 
					$T_s:=T_\mu^0:\mathcal{O}_\mu\rightarrow\mathcal{O}_{0}$ which are left and right adjoint to each other.
					Their composition $\Theta_s=T_sT^s$ is an endofunctor of $\mathcal{O}_0$. 
					Let $\D$ be the Hecke category
					associated with $(W,S)$ and the realization $\mathfrak{h}$. Then one has a functor
					$\D\rightarrow \End(\Proj\mathcal{O}_0)$ sending $B_s$ to $\Theta_s$. This was shown 
					by Soergel \cite{Soe,Soe_HC} in terms of his bimodules, by embedding the subcategory of projectives 
					inside the category of modules over the coinvariant ring $C=R/R\cdot R^W_+$.
		\end{rmk}
		\section{Rouquier complexes}
		In this section we introduce Rouquier complexes in the homotopy category of the 
		Hecke category.
		\subsection{Notation for categories of complexes}
		Given a $\Bbbk$-linear additive category $\mathscr{C}$, let 
		$\C(\CC)$ and $\K(\CC)$ be the category of complexes and the homotopy category associated to $\CC$, respectively.

		Recall that if $\CC$ has a monoidal structure, given by $\otimes$, 
		then this extends to $\C(\CC)$ as follows.
		Let $\Ao,\Bo\in \C(\CC)$, then:
		\begin{equation}\label{eq_grobjdef}
			(\Ao\otimes \Bo)^q:=\bigoplus_{i+j=q} A^i\otimes B^j,		
		\end{equation}
		and the differential map, restricted to $A^i\otimes B^j$ is:
		\begin{equation}\label{eq_signconv}
			d_A\otimes \id + (-1)^i \id \otimes d_B.
		\end{equation}
		One easily defines tensor product of morphisms. The same definition
		gives a monoidal structure also on $\K(\CC)$. Let $\Cb(\CC)$ and $\Kb(\CC)$ denote the corresponding
		bounded full subcategories.
		\subsection{Definition of Rouquier complexes}\label{subs_Roucom}
		Consider the homotopy category $\K(\D)$ with the monoidal structure induced from that of $\D$. 
		The monoidal unit is the object $\un$ concentrated in degree 0.
			
		Consider the \emph{standard} and \emph{costandard} complexes $F_s$ and $F_s^{-1}$ 
		from \eqref{eq_stdcostd}.
		If $\sigma=\sigma_{s}^{\pm 1}\in \Sigma$, then let 
		$F_\sigma$ denote $F_s^{\pm 1}$.
		Then, for any braid word $\uom=\sigma_1 \sigma_2 \dots \sigma_n$, with $\sigma_i\in \Sigma$,
		we put:
		\begin{equation}\label{eq_defrouq}
			F_{\uom}^\bullet:=F_{\sigma_1}\otimes F_{\sigma_2}\otimes \dots \otimes F_{\sigma_n}.		
		\end{equation}
		In the sequel we will often omit the symbol $\otimes$.
		These objects are called
		\emph{Rouquier complexes} and the subcategory $\mathscr{B}_W$ that they form is called \emph{2-braid group}. 
		\begin{rmk}\label{rmk_orRou}
		These complexes were introduced in \cite{Rou_cat} in terms of Soergel bimodules 
		to study natural transformations between the endofunctors induced by braid group actions on categories. 
		As an example, in the setting of Remark \ref{rmk_onSoergelcat}, 
		consider the complexes
		\begin{align*}
			\tilde{F}_s=\dots\rightarrow 0\rightarrow 0\rightarrow &\,\Theta_s\rightarrow \un\rightarrow 0\rightarrow \dots\\
			\tilde{F}_s^{-1}=\dots\rightarrow 0\rightarrow\un \rightarrow &\,\Theta_s\rightarrow 0\rightarrow 0\rightarrow\dots
		\end{align*}
		where $\Theta_s$ is	in degree 0 for both and the only non-zero morphisms are
		given by the counit and unit of the two adjunctions between $T_s$ and $T^s$. 
		One can show that they are mutually inverse 
		self-equivalences of the category $\Kb(\Proj\mathcal{O}_0)=\mathcal{D}^{\mathrm{b}}(\mathcal{O}_0)$, and that 
		the assignment $\sigma_s\mapsto\tilde{F}_s$ defines an action of $B_W$ on this category \cite[Theorem 10.4]{Rou_cat}. Furhtermore
		one can upgrade this to a functor from $\mathscr{B}_W$ to the category of self-equivalences
		of $\mathcal{D}^{\mathrm{b}}(\mathcal{O}_0)$ \cite[Theorem 10.5]{Rou_cat}. Hence morphisms between 
		Rouquier complexes give natural transformations between these endofunctors.
		\end{rmk}
		\begin{rmk}
		In the geometric setting (see Remark \ref{rmk_incarn} \ref{item_geom}), the 
		homotopy category $\Kb(\D)$ was considered by Achar and Riche 
		\cite[\S 3.5]{AchRic_modII} as a modular version of the equivariant mixed derived category on 
		the (affine) flag variety. The standard and costandard complexes $F_s$ and $F_s^{-1}$ correspond to
		the standard and costandard sheaves $\stdmix_{s}$ and $\cosmix_{s}$ in that setting.
		\end{rmk}
		\subsection{Basic properties of Rouquier complexes}
		These properties were first proved by Rouquier \cite{Rou_cat} in the 
		language of Soergel bimodules. A sketch of proof is also given in \cite{Elias}. One can find a 
		diagrammatic argument in \cite{Mal_rou}.
		\begin{pro}\label{pro_braidRou}
			One has the following.
			\begin{enumerate}
			 	\item\label{item_propinvers} Let $s\in S$, then $F_s F_s^{-1}\cong F_s^{-1}F_s\cong \un$.
				\item\label{item_propbraid} Let $s,t\in S$ with $m_{st}<\infty$, then
					\[
						\underbrace{F_s F_tF_sF_t \cdots}_{m_{st} \text{ times}} \cong %
							\underbrace{F_t F_sF_tF_s \cdots}_{m_{st} \text{ times}}
					\]
			\end{enumerate}
			Hence, for each pair of braid words $\uom_1$, $\uom_2$ expressing
			the same element $\omega\in B_W$, there is an isomorphism $F_{\uom_1}^\bullet\cong F_{\uom_2}^\bullet$.
			Furthermore:
			\begin{enumerate}[resume]
				\item\label{item_propcanonic}(Rouquier Canonicity)
					for each $\uom_1$ and $\uom_2$ as above, we have 
					\[
						\Hom(F_{\uom_1}^\bullet,F_{\uom_2}^\bullet)\cong R,
					\]
					and one can chose isomorphisms 
					$\gamma_{\uom_1}^{\uom_2}$ from $F_{\uom_1}^\bullet$ to $F_{\uom_1}^\bullet$ so that the system 
					$\{\gamma_{\uom_1}^{\uom_2}\}_{\uom_1,\uom_2}$ is transitive.
			\end{enumerate}
		\end{pro}
		\begin{rmk}\label{rmk_faithful}
			Thanks to these properties,	the Rouquier complexes $F_{\omega}$'s associated to $\omega\in B_W$ are 
			well defined up to a canonical isomorphism and they form a quotient
			of the braid group.	Rouquier conjectured that this quotient is just $B_W$ itself.
			This \emph{faithfulness} of the 2-braid group was shown in type $A$ by Khovanov and Seidel \cite{KhoSei}, in simply laced 
			finite type by Brav and Thomas \cite{BraTho}, and in all finite types by Jensen \cite{Jen}.
		\end{rmk}		
		The complex \eqref{eq_defrouq}, seen as an object of $\C(\D)$ is called a 
		\emph{standard rerpesentative} for $F_\omega$.
		When the category $\D$ is Krull-Schmidt, one can 
		get rid of a maximal null-homotopic summand of each $F_{\uom}^\bullet$ and
		obtain a summand with no null-homotopic factor, called
		the \emph{minimal subcomplex} for $F_{\omega}$. 
		One can show that this is unique up to isomorphism (and not just homotopy equivalence): see \cite[\S 6.1]{EWHodge}.
		The minimal subcomplexes of Rouquier complexes are hard to find in general and for arbitrary 
		coefficients, this notion is not well defined. 
			
		For $\omega$ positive we will now describe a representative 
		for $F_{\omega}$ which is simpler than $F_{\uom}^\bullet$, and available 
		without restrictions on $\kk$. This can be considered an intermediate step towards a minimal subcomplex.

		First we describe standard representatives more precisely in this case.
		\subsection{Standard representatives}\label{subs_heckcat_diagrams}
		Let $\omega\in B_W$ be a positive braid expressed by 
		$\uom=\sigma_{s_1}\dots \sigma_{s_n}$.
		Let $\uw$ be the Coxeter word $s_1\dots s_n$.
		Let:
		\begin{equation}
			F^\bullet_{\uw}:=F_{s_1}\otimes \dots \otimes F_{s_{n}}
		\end{equation}
		be the corresponding standard representative for $F_\omega$.
		
		For a given subexpression $\boi$ of $\uw$, let $\uw_{\boi}$ denote the 
		subword corresponding to it and let $B_{\boi}:=B_{\uw_{\boi}}$ be 
		the associated Bott-Samelson object.
			Let also $q_{\boi}=\ell(\uw)-\ell(\uw_{\boi})$.
			Iterating \eqref{eq_grobjdef}, the $q$-th graded piece $F_{\uw}^q$ 
			of $\Roudg{\uw}$ is:
			\begin{equation}\label{eq_gradRoudg}
				F_{\uw}^q=\bigoplus_{j_1+\dots+j_n=q}F_{s_1}^{j_1} \otimes \dots \otimes F_{s_n}^{j_n} =\bigoplus_{\substack{ \boi\in\{0,1\}^{\ell(\uw)} \\ q_{\boi} = q }} B_{\boi} (q).			
			\end{equation}

			One can then compute the components of the differential map, according to
			\eqref{eq_signconv}. The only nonzero components $B_{\boi}\rightarrow B_{\boj}$ are those 
			for which $\boj$ is obtained from $\boi$ by turning a 1 to a 0
			(so that $q_{\boj}=q_{\boi}+1$).
			In this case the component is:
			\begin{equation}\label{eq_diffdiagsign}
				(-1)^a	
					\begin{tikzpicture}[x=0.3cm,y=0.4cm,baseline=0.36cm,every node/.style={font=\scriptsize}]
						\draw[gray] (0.5,0)--(9.5,0);\draw[gray] (0.5,2)--(9.5,2);
						\foreach \i in {1,2,4,6,8,9}{%
								\draw[violet] (\i,0)--(\i,2);
							}
						\draw[violet] (5,0)--(5,1);\fill[violet] (5,1) circle (1.5pt);
						\node[font=\normalsize] at (3,1) {\dots}; \node[font=\normalsize] at (7,1) {\dots};
					\end{tikzpicture}
			 \end{equation}
			where $a$ is the number of 0's in $\boi$ preceding the changed symbol.

			Then the complex $F_{\uw}^\bullet$ has the form 
			of a cube of dimension $\ell(\uw)$: the vertices correspond to the 
			Bott-Samelson objects obtained from all possible subexpressions of $\uw$ and the edges are the arrows 
			describing the components of the differential map.
			\begin{exa}\label{exa_cube}
				Let $\re{s},\bl{t}\in S$ and 
				$\uom=\sigma_s\sigma_s\sigma_t$. 
				The complex $\Roudg{\uom}=F_sF_sF_t$ 
				is the following:
				
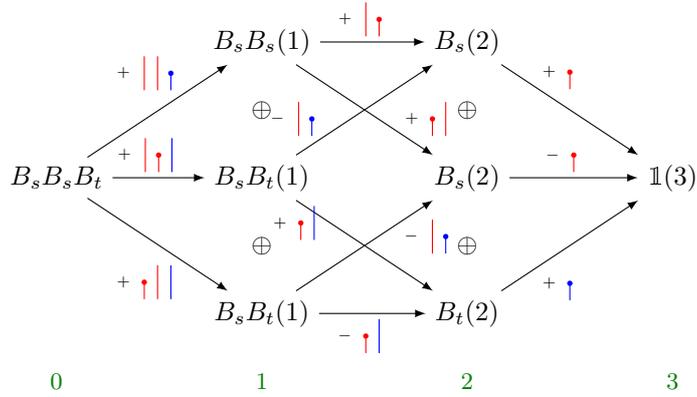
\begin{center}
					\begin{tikzpicture}[x=2.7cm,y=0.9cm,every pic/.style={x=1cm,scale=0.5,font=\tiny,-}]
						\node (s1s2t) at (0,0) {$B_sB_sB_t$};
						\node (s1s2) at (1,2) {$B_sB_s(1)$};
						\node (s1t) at (1,0) {$B_sB_t(1)$};
						\node (s2t) at (1,-2) {$B_sB_t(1)$};
						\node (s1) at (2,2) {$B_s(2)$};
						\node (s2) at (2,0) {$B_s(2)$};
						\node (t) at (2,-2) {$B_t(2)$};
						\node (n1) at (3,0) {$\un(3)$};
						\draw[-latex] (s1s2t) to %
							pic[midway,yshift=.5cm]{topdot={sst}{3}{+}} (s1s2);%
						\draw[-latex] (s1s2t) to%
							pic[midway,yshift=.3cm]{topdot={sst}{2}{+}} (s1t);%
						\draw[-latex] (s1s2t) to%
							pic[midway,yshift=-.5cm]{topdot={sst}{1}{+}} (s2t);%
						\draw[-latex] (s1s2) to%
							pic[midway,yshift=.3cm]{topdot={ss}{2}{+}} (s1);%
						\draw[-latex] (s1s2) to%
							pic[pos=.6,xshift=.8cm]{topdot={ss}{1}{+}} (s2);%
						\draw[-latex] (s1t) to%
							pic[pos=.4,xshift=-.6cm]{topdot={st}{2}{-}} (s1);%
						\draw[-latex] (s1t) to%
							pic[near start,xshift=-.3cm]{topdot={st}{1}{+}} (t);%
						\draw[-latex] (s2t) to%
							pic[pos=.6,xshift=.8cm]{topdot={st}{2}{-}} (s2);%
						\draw[-latex] (s2t) to%
							pic[midway,yshift=-.3cm]{topdot={st}{1}{-}} (t);%
						\draw[-latex] (s1) to%
							pic[midway,yshift=.5cm]{topdot={s}{1}{+}} (n1);%
						\draw[-latex] (s2) to%
							pic[midway,yshift=.3cm]{topdot={s}{1}{-}} (n1);%
						\draw[-latex] (t) to%
							pic[midway,yshift=-.5cm]{topdot={t}{1}{+}} (n1);%
						\node[green,font=\small] at (0,-3){0}; \node[green,font=\small] at (1,-3){1};%
						\node[green,font=\small] at (2,-3){2}; \node[green,font=\small] at (3,-3) {3};
						\node at (1,1) {$\oplus$}; \node at (1,-1) {$\oplus$};
						\node at (2,1) {$\oplus$}; \node at (2,-1) {$\oplus$};
					\end{tikzpicture}		
				\end{center}			
			\end{exa}	 
			Next we describe an isomorphic \emph{twisted} representative, with a different, and more handy, sign convention.
			\subsection{Sign convention}\label{subs_signconv}
			As described by Elias \cite[\S 4.7]{Elias}, one can choose a different sign convention for the differential in $\Roudg{\uom}$, 
			giving an isomorphic complex. One considers the following 
			twisted tensor product of complexes. 
			\begin{dfn}
				Let $\CC$ be an monoidal additive category and $A^\bullet,B^\bullet$ be complexes in $\C(\CC)$. Then
				the twisted tensor product $A^\bullet\dot{\otimes}B^\bullet$ is the complex with:
				\begin{equation}
					(A^\bullet\dot{\otimes} B^\bullet)^q =\bigoplus_{i+j=q} A^i\otimes B^j,
				\end{equation}
				and with differential, restricted to each object $A^i\otimes B^j$, given by:
				\begin{equation}\label{eq_twisted}
					d_A \otimes \id_B + (-1)^{i+1} \id_A\otimes d_B.
				\end{equation}
			\end{dfn}
			The only difference is the sign of the components of the differential. And it is easy to prove that the complex
			$A^\bullet\dot{\otimes}B^\bullet$ is isomorphic to $A^\bullet\otimes B^\bullet$.
			Elias proves \cite[Lemma 4.16]{Elias} that $\dot{\otimes}$ is associative for certain 
			types of complexes, including the standard and costandard 
			complexes $F_s$ and $F_s^{-1}$ (for all $s$). In particular 
			the product
			\begin{equation}
				E^\bullet_{\uw} := F_{s_1}\dot{\otimes}F_{s_2} \dot{\otimes} \dots \dot{\otimes} F_{s_n}
			\end{equation}
			is well defined and gives a complex isomorphic to $\Roudg{\uw}$, that we call the \emph{twisted representative} for $F_{\omega}$. 
			The graded pieces are the same as in $\Roudg{\uw}$:
			\begin{equation}\label{eq_grpiecE}
				E_{\uw}^q = \bigoplus_{\substack{ \boi\in\{0,1\}^{\ell(\uom)} \\ q_{\boi} = q }} B_{\boi} (q),
			\end{equation}
			and the components of the differential
			are still as in \eqref{eq_diffdiagsign} but $a$ 
			is now the number of 1's (instead of the number of 0's) 
			preceding the changed symbol. In other words $a$ is the 
			number of letters in $\uw_{\boi}$ preceding the one being canceled or also the number of
			strands on the left of the dot in \eqref{eq_diffdiagsign}.
			\begin{exa}
				Consider $\uom=\sigma_s\sigma_s\sigma_t$ as in
				Example \ref{exa_cube} the twisted representative
				$E^\bullet_{\uw}=F_s\dot{\otimes} F_s\dot{\otimes} F_t$ 
				is the following:
				
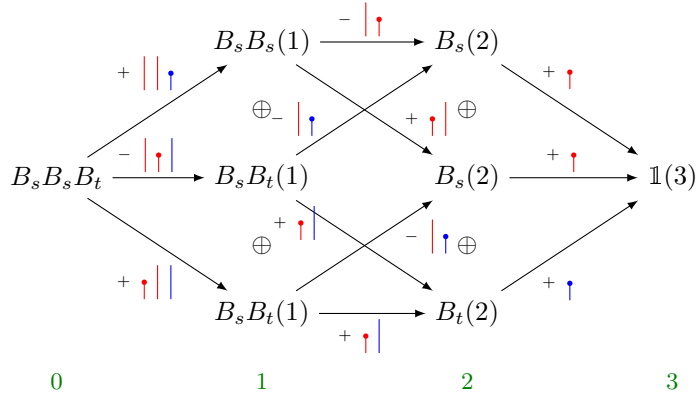
\begin{center}
					\begin{tikzpicture}[x=2.7cm,y=0.9cm,every pic/.style={x=1cm,scale=0.5,font=\tiny,-}]
						\node (s1s2t) at (0,0) {$B_sB_sB_t$};
						\node (s1s2) at (1,2) {$B_sB_s(1)$};
						\node (s1t) at (1,0) {$B_sB_t(1)$};
						\node (s2t) at (1,-2) {$B_sB_t(1)$};
						\node (s1) at (2,2) {$B_s(2)$};
						\node (s2) at (2,0) {$B_s(2)$};
						\node (t) at (2,-2) {$B_t(2)$};
						\node (n1) at (3,0) {$\un(3)$};
						\draw[-latex] (s1s2t) to %
							pic[midway,yshift=.5cm]{topdot={sst}{3}{+}} (s1s2);%
						\draw[-latex] (s1s2t) to%
							pic[midway,yshift=.3cm]{topdot={sst}{2}{-}} (s1t);%
						\draw[-latex] (s1s2t) to%
							pic[midway,yshift=-.5cm]{topdot={sst}{1}{+}} (s2t);%
						\draw[-latex] (s1s2) to%
							pic[midway,yshift=.3cm]{topdot={ss}{2}{-}} (s1);%
						\draw[-latex] (s1s2) to%
							pic[pos=.6,xshift=.8cm]{topdot={ss}{1}{+}} (s2);%
						\draw[-latex] (s1t) to%
							pic[pos=.4,xshift=-.6cm]{topdot={st}{2}{-}} (s1);%
						\draw[-latex] (s1t) to%
							pic[near start,xshift=-.3cm]{topdot={st}{1}{+}} (t);%
						\draw[-latex] (s2t) to%
							pic[pos=.6,xshift=.8cm]{topdot={st}{2}{-}} (s2);%
						\draw[-latex] (s2t) to%
							pic[midway,yshift=-.3cm]{topdot={st}{1}{+}} (t);%
						\draw[-latex] (s1) to%
							pic[midway,yshift=.5cm]{topdot={s}{1}{+}} (n1);%
						\draw[-latex] (s2) to%
							pic[midway,yshift=.3cm]{topdot={s}{1}{+}} (n1);%
						\draw[-latex] (t) to%
							pic[midway,yshift=-.5cm]{topdot={t}{1}{+}} (n1);%
						\node[green,font=\small] at (0,-3){0}; \node[green,font=\small] at (1,-3){1};%
						\node[green,font=\small] at (2,-3){2}; \node[green,font=\small] at (3,-3) {3};
						\node at (1,1) {$\oplus$}; \node at (1,-1) {$\oplus$};
						\node at (2,1) {$\oplus$}; \node at (2,-1) {$\oplus$};
					\end{tikzpicture}		
				\end{center}			
			\end{exa}	  			 			
			Finally, we describe 
			a new complex $R^\bullet_{\uw}$ that we will obtain from 
			$E^\bullet_{\uw}$ by Gaussian elimination (see next section).
			\subsection{Reduced representatives}\label{subs_reducedcom}
			If $s\in S$ and $\ux$ is the monotonous word consisting 
			of $n$ repetitions of the letter $s$, we set
			\[
				\worsum{\ux}:=B_s(-n+1).
			\]
			Iterating Example \ref{exa_basicdeco}, one sees that this is a summand inside $B_s^{\otimes n}=B_{\ux}$ (see also \eqref{eq_iteratdeco} below). If $\ux$ is any Coxeter word, first write it as 
			\begin{equation}\label{eq_mondecomp}
				\ux_1\ux_2\dots \ux_k
			\end{equation}
			where each $\ux_i$ is monotonous but $\ux_i\ux_{i+1}$ is not.
			Then set
			\[
				\worsum{\ux}:=\worsum{\ux_1}\worsum{\ux_2}\cdots \worsum{\ux_k}\underset{\oplus}{\subset} B_{\ux}.
			\]
			We also set $\worsum{\emptyset}:=\un$. In other words, $C_{\ux}=B_{\ux^*}(\ell(\ux^*)-\ell(\ux))$, where $\ux^*$ is obtained from $\ux$
			by contracting each monotonous subsequence to a single letter.
			\begin{exa}
				Let $\ux=sssttusuu$, then:
				\begin{align*}
					&\ux_1=sss,& &\ux_2=tt,& &\ux_3=u,& &\ux_4=s,& &\ux_5=uu.
				\end{align*}
				Hence $\ux^*=stusu$ and $\worsum{\ux}=B_s(-2)\otimes B_t(-1) \otimes B_u \otimes B_s \otimes B_u(-1)=B_{\ux^*}(-4)$.
			\end{exa}
			Now, for $\ux\preceq\uw$, let 
			$q_{\ux}=\ell(\uw)-\ell(\ux)$. The $q$-th graded piece $R^q_{\uw}$ of $R^\bullet_{\uw}$ is:
			\begin{equation}
				R^q_{\uw}=\bigoplus_{\substack{ \ux\preceq \uw \\ q_{\ux}=q }} \worsum{\ux} (q).			
			\end{equation}
			We have to describe the differential map $d$.
			For any $s\in S$ and $k$ a non-negative integer, let us introduce the maps:
			\[
				d_{s,k}:=
				\begin{cases}
					\def\hdot{.32cm}
					\,\,
					\begin{tikzpicture}
						\pic[red] at (0,0) {dot=\hdot};
					\end{tikzpicture} & \text{if $k=0$,}\\
					\begin{tikzpicture}[baseline=.6*\hdot]
						\draw[red] (0,0)--(0,1.7*\hdot);
						\node[anchor=east] at (0,.85*\hdot) {$\delta_{\re{s}}$};
					\end{tikzpicture}
					\hphantom{\delta_s}
					-
					\hphantom{\delta_s}
					\begin{tikzpicture}[baseline=.6*\hdot]
						\draw[red] (0,0)--(0,1.7*\hdot);
						\node[anchor=west] at (0,.85*\hdot) {$\delta_{\re{s}}$};
					\end{tikzpicture} 
					=\,
					\begin{tikzpicture}[baseline=.6*\hdot]
						\draw[red] (0,0)--(0,1.7*\hdot);
						\node[anchor=east] at (0,.85*\hdot) {$\alpha_{\re{s}}$};
					\end{tikzpicture}
					\hphantom{\delta_s}- \hphantom{\delta_s}
					\begin{tikzpicture}[baseline=.6*\hdot]
						\draw[red] (0,0)--(0,.6*\hdot);\fill[red] (0,0.6*\hdot) circle (1.5pt);
						\draw[red] (0,1.7*\hdot)--(0,1.1*\hdot);\fill[red] (0,1.1*\hdot) circle (1.5pt);
					\end{tikzpicture} & \text{if $k>0$ is odd,}
					\\[1em]
					\begin{tikzpicture}[baseline=.6*\hdot]
						\draw[red] (0,0)--(0,1.7*\hdot);
						\node[anchor=east] at (0,.85*\hdot) {$\delta_{\re{s}}$};
					\end{tikzpicture}
					\hphantom{\delta_s}-\hphantom{\delta_s}
					\begin{tikzpicture}[baseline=.6*\hdot]
						\draw[red] (0,0)--(0,1.7*\hdot);
						\node[anchor=west] at (0,.85*\hdot) {${\re{s}}(\delta_{\re{s}})$};
					\end{tikzpicture}
					=\,
					\begin{tikzpicture}[baseline=.6*\hdot]
						\draw[red] (0,0)--(0,.6*\hdot);\fill[red] (0,0.6*\hdot) circle (1.5pt);
						\draw[red] (0,1.7*\hdot)--(0,1.1*\hdot);\fill[red] (0,1.1*\hdot) circle (1.5pt);
					\end{tikzpicture} & \text{if $k>0$ is even.}
				\end{cases}  			
			\]
			So, for any $n$, we can view these as morphisms:
			\begin{align}
				& d_{s,0}:B_s(n)\rightarrow \un(n+1),&
				& d_{s,k}:B_s(n)\rightarrow B_s(n+2)\quad\text{(for $k>0$)}.
			\end{align}
			We define $d$ by specifying its components:
			\begin{equation}\label{eq_compreddiff}
				d_{\ux}^{\uz}:\worsum{\ux}(q_{\ux})\rightarrow \worsum{\uz}(q_{\uz}),			
			\end{equation}
			for $\ux,\uz\preceq \uw$ and $\ell(\ux)=\ell(\uz)+1$, so that $q_{\ux}+1=q_{\uz}$. 			
			We declare the map \eqref{eq_compreddiff} to be nonzero only 
			if $\uz\preceq \ux$, which means that $\uz$ is obtained 
			from $\ux$ by eliminating one letter.
			In this case, one can write $\ux$ and $\uz$ in the forms:
			\begin{align}\label{eq_formxx}
				&\ux=\uw_1 \underbrace{sss\dots s}_{k+1} \uw_2,& 
				&\uz=\uw_1 \underbrace{ss\dots s}_{k} \uw_2,
			\end{align}
			where $\uw_1$ and $\uw_2$ are (possibly empty) Coxeter words such that $\uw_1$ does not
			end with $s$ and $\uw_2$ does not start with $s$. 
			So $C_{\ux}=C_{\uw_1}B_s(-k)C_{\uw_2}$ and we set:
			\begin{equation}\label{eq_diffred}
				d_{\ux}^{\uz}:=(-1)^{\ell(\uw_1)} \id_{\worsum{\uw_1}}\otimes d_{s,k} \otimes \id_{\worsum{\uw_2}},			
			\end{equation}
			If $k=0$ and $\uw_1$ ends with same letter as $\uw_2$ starts, then we also compose on top 
			with the corresponding trivalent vertex:
			
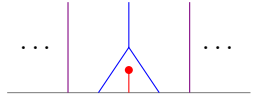
\[
				\begin{tikzpicture}[y=.3cm,x=.4cm,baseline=.5cm]
					\draw[gray] (-1,0)--(7,0);
					\draw[blue] (2,0)--(3,2)--(4,0);\draw[blue](3,2)--(3,4);
					\draw[violet] (1,0)--(1,4);\draw[violet] (5,0)--(5,4);
					\draw[red] (3,0)--(3,1); \fill[red] (3,1) circle (1.5pt);
					\node at (0,2) {\dots};\node at (6,2) {\dots};
				\end{tikzpicture}			
			\]
			So the target of this morphism is always $C_{\uz}$.
			\begin{exa}
				Let $\uw=\bl{t}\re{ss}\bl{t}$, then $F_{\uw}$ is
				\[
					\begin{tikzcd}[row sep=tiny]
																				& 															& \worsum{tt}(2)\ar[rdd]			& 						&						\\
																				& \worsum{tst}(1)\ar[ru,"d_2"]\ar[rd]\ar[rddd]	& \oplus										&						&						\\
																				& \oplus													& \worsum{ts}(2)\ar[r]\ar[rdd]		& \worsum{t}(3)\ar[rd]	&						\\
						\worsum{tsst}\ar[ruu,"d_1"]\ar[r]\ar[rdd]	& \worsum{tss}(1)\ar[ru]\ar[rddd]				& \oplus 										& \oplus				& \worsum{\emptyset}(4) 	\\
																				& \oplus													& \worsum{st}(2)\ar[r]\ar[ruu]		& \worsum{s}(3)\ar[ru]	& 						\\
																				& \worsum{sst}(1)\ar[ru]\ar[rd]					& \oplus										& 						&						\\
																				& 															& \worsum{ss}(2)\ar[ruu]			& 						&
					\end{tikzcd}
				\]
				We describe the arrows $d_1$ and $d_2$, as an example, 
				and we leave the others to the reader.
				We have $\worsum{tsst}=B_tB_sB_t(-1)$ and $\worsum{tst}(1)=B_tB_sB_t(1)$. We are canceling the letter $s$
				and passing from two occurrences to one. Here $\uz_1=t$ so the sign is negative.
				Then the morphism $d_1$ is:
				\[	
					-\begin{tikzpicture}[x=.5cm,baseline=.4cm]
						\draw[gray] (0.5,0)--(3.5,0); \draw[gray] (.5,1)--(3.5,1);
						\draw[blue] (1,0)--(1,1); \draw[blue] (3,0)--(3,1);
						\draw[red] (2,0)--(2,1);
						\node at (1.5,0.5) {$\delta_{\re{s}}$};
					\end{tikzpicture}
					-
					\begin{tikzpicture}[x=.5cm,baseline=.4cm]
						\draw[gray] (0.5,0)--(3.5,0); \draw[gray] (.5,1)--(3.5,1);
						\draw[blue] (1,0)--(1,1); \draw[blue] (3,0)--(3,1);
						\draw[red] (2,0)--(2,1);
						\node at (2.5,0.5) {$\delta_{\re{s}}$};
					\end{tikzpicture}
				\]
				Next, we have $\worsum{tt}(2)=B_t(1)$ and the morphism $d_2$ is eliminating the only 
				$s$ from $\worsum{tst}(1)$. 
				Again $\uz_1=t$ so the sign is negative.
				Notice also that the two adjacent letters are both $t$'s, hence $d_2$ is:
				\[	
					-
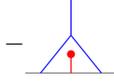
					\begin{tikzpicture}[x=.4cm,y=.5cm,baseline=.3cm]
						\draw[gray] (0.5,0)--(3.5,0); \draw[gray] (0.5,2)--(3.5,2);
						\draw[blue] (1,0)--(2,1)--(2,2); \draw[blue] (3,0)--(2,1);
						\draw[red] (2,0)--(2,0.5);\fill[red] (2,0.5) circle (1.5pt);
					\end{tikzpicture}
				\]
			\end{exa}
%
%
			Here is the statement of the main result of this paper.
			\begin{thm}\label{thm_reduce}
				Let $\uw$ be a Coxeter word. Then $R^\bullet_{\uw}$ is a 
				complex and a summand of $E^\bullet_{\uw}$ such that the inclusion and projection maps are
				inverse homotopy equivalences. In particular $\Roudg{\uw}\cong E^\bullet_{\uw} \simeq R^\bullet_{\uw}$.
			\end{thm}
			\begin{exa}
				If $\uw=\underline{s}^n$, the theorem gives the minimal complex for $F_s^n$. See
				\cite[Exercise 19.29]{EMTW}.
			\end{exa}
			\begin{rmk}
				If $\uw'\preceq\uw$, then clearly $\Roudg{\uw'}$ is, up to shift, a subcomplex of $\Roudg{\uw}$. The same still holds for
				$R^\bullet_{\uw'}$ and $R^\bullet_{\uw}$. In fact the definition of $F_{\uw}$ only depends on the poset of 
				$\{\ux\preceq\uw\}$ of the subwords of $\uw$,
				and the differential preserves the order ($C_{\ux}$ has an arrow towards $C_{\ux'}$ only if $\ux'\preceq\ux$).
				Then it is sufficient to notice that $\{\ux\preceq\uw'\}$ is a subposet of $\{\ux\preceq\uw\}$.			
			\end{rmk}
			\begin{rmk}
				One can also define reduced representatives for \emph{negative} Rouquier complexes (i.e., corresponding to braids in the monoid generated by $\Sigma^-$. 
				The description of standard and reduced representatives 
				is entirely symmetric. One only needs to reverse the arrows, flip the diagrams upside down and 
				add a minus sign to the cohomological degrees. Then an analogue of Theorem \ref{thm_reduce} holds.
			\end{rmk}
			The next two sections are devoted to the proof of Theorem \ref{thm_reduce}. First we introduce our main tool:
			Morse theoretical Gaussian elimination. Then we explain how to use it in our case.
			\section{Morse Theoretical Gaussian Elimination}\label{sec_morsegauss}
			To obtain reduced representatives of Rouquier complexes we will use
			repeated Gaussian elimination imitating the reduction of CW
			complexes via Discrete Morse Theory, in the sense of Forman \cite{For}.
			This section is just a rephrasing of a work of Sk\"oldberg \cite{Sko},
			in terms of Gaussian elimination of complexes over additive categories, in the finite case.
			Some results about simultaneous Gaussian eliminations are also described 
			by Elias \cite{Elias}.
			\subsection{Usual Gaussian elimination}
			Let $\CC$ be an additive category and let $\C(\CC)$ denote the corresponding 
			category of complexes. 
			\begin{dfn}
				A summand of a complex in $\C(\CC)$ is called \emph{Gaussian} if the corresponding inclusion
				and projection maps are inverse homotopy equivalences.
			\end{dfn}
			First, let us recall the ``one-step'' Gaussian elimination for complexes.
			The following was first pointed out by Bar-Natan \cite{BarNat}.
			\begin{lem}\label{lem_Gausselimination}
				Consider a complex $A^\bullet$ in $\C(\CC)$ of the form:
				\begin{equation*}
					\begin{tikzcd}[ampersand replacement=\&,baseline=.15cm]
						\dots\arrow{r} \& C^{q-1} \arrow{r}{ \begin{psmallmatrix}
										e \\ f
				 					 \end{psmallmatrix} }		\& C^q \oplus E \arrow{r}{ \begin{psmallmatrix}
																						a	& b \\
																						c	& d 
																					\end{psmallmatrix} } \& C^{q+1} \oplus E' \arrow{r}{ 	\begin{psmallmatrix}
																																		g	& h \\
																																	\end{psmallmatrix} } 	\& C^{q+2} \arrow{r} \& \dots
					\end{tikzcd}
				\end{equation*}
				and suppose that $d: E\rightarrow E'$ is an isomorphism in $\CC$. 
				Then the following:
				\[
					\begin{tikzcd}[column sep=large]
					 \dots \ar[r] & C^{q-1} \ar[r,"e"] & C^q \ar[r,"a-bd^{-1}c"] & C^{q+1} \ar[r,"g"] & C^{q+2} \ar[r] & \dots
					\end{tikzcd}
				\]
				is a complex and a Gaussian summand of $A^\bullet$.
				\end{lem}
				\begin{proof}
					Consider the projection morphism $\pi$ given by:
					\begin{equation*}
						\begin{tikzcd}[ampersand replacement=\&,baseline=.15cm]
							\dots\arrow{r} \& C^{q-1} \arrow{r}
										\ar[d,equal]		\& C^{q}\oplus E \arrow{r}
																						\arrow{d}{\begin{psmallmatrix}
																									\id, & 0
																								\end{psmallmatrix}}
																												\& C^{q+1}\oplus E' \arrow{r}
																																\arrow{d}{\begin{psmallmatrix}
																																	\id, & -bd^{-1}
																																\end{psmallmatrix}} 	\& C^{q+2}\ar[d,equal] \arrow{r} \& \dots \\
							\dots\arrow{r} \& C^{q-1} \ar[r]		\& C^{q} \ar[r]	\& C^{q+1} \ar[r]	\& C^{q+2} \arrow{r} \& \dots
						\end{tikzcd}			
					\end{equation*}
					and the inclusion morphism $\iota$ described by:
					\[
						\begin{tikzcd}[ampersand replacement=\&,baseline=.15cm]
							\dots \arrow{r} \& C^{q-1} \arrow{r}
					 					 \& C^{q}\oplus E \arrow{r}
																					\& C^{q+1}\oplus E' \arrow{r}
																											\& C^{q+2} \arrow{r} \& \dots \\
							\dots \arrow{r} \& C^{q-1} \ar[r]\ar[u,equal]		\& C^{q} \ar[r]	\arrow{u}{\begin{psmallmatrix}
																																	\id \\ -d^{-1}c
																																\end{psmallmatrix}}	\& C^{q+1} \ar[r]	\arrow{u}{\begin{psmallmatrix}
																																												\id \\ 0
																																											\end{psmallmatrix}} 	\& C^{q+2} \ar[u,equal] \arrow{r} \& \dots 
						\end{tikzcd}			
					\]
					It is easy to check that $\pi$ and $\iota$ are maps of complexes and that $\pi\iota=\id$.
					The complementary summand:
					\[
						\begin{tikzcd}
							\dots \ar[r] & 0 \ar[r] & E \ar[r,"d"] & E' \ar[r] & 0 \ar[r] &\dots
						\end{tikzcd}
					\]
					is contractible, so the idempotent $\iota\pi$ is homotopy equivalent to the identity.
				\end{proof}
				Now we give an equivalent version of this result in terms of based complexes.
				A \emph{based complex} is an object $K^\bullet$ in 
				$\C(\CC)$ with a given decomposition of its graded pieces:
				\[
					K^q = \bigoplus_{\sigma\in I_q} K_\sigma,
				\]
				where the $I_q$'s are disjoint index sets. Let $V=\cup_q I_q$.
				For $\sigma,\tau\in V$ let $d_\sigma^\tau$ be the component 
				$K_\sigma\rightarrow K_\tau$ of the differential map $d$ of $K^\bullet$
				(we set $d_\sigma^\tau=0$ unless 
				$\sigma\in I_q$ and $\tau\in I_{q+1}$, for some $q$).
				\begin{lem}\label{lem_2gauss}
					Let $K^\bullet$ be a based complex with differential map $d$. Suppose that, 
					for some $\alpha,\beta\in V$, the component 
					$d_{\alpha}^\beta$ of the differential 
					is an isomorphism. 
					Let: 
					\[
						\tilde{K}^q:=\bigoplus_{\sigma\in I_q\setminus \{\alpha,\beta\}} K_\sigma, \quad \text{for all $q$}.
					\]
					So $\tilde{K}^q$ and $K^q$ differ only at the two (consecutive) degrees of $\alpha$ and $\beta$.
					Let $\tilde{d}$ be defined componentwise as: 
					\begin{align}\label{eq_newdiff}
						&\tilde{d}_{\sigma}^\tau=
							d_{\sigma}^\tau - d_\alpha^\tau (d_\alpha^\beta)^{-1} d_\sigma^\beta,& 
						& \text{for all } \sigma,\tau \in V\setminus\{\alpha,\beta\}.
					\end{align}
					Then $\tilde{K}^\bullet$, endowed with $\tilde{d}$, is a complex and a Gaussian summand
					of $K^\bullet$.
				\end{lem}		
				\begin{proof}
					It suffices to apply Lemma \ref{lem_Gausselimination}
					with:
					\begin{align*}
						& E=K_\alpha,& &E'=K_\beta& &\text{and}& 
						C^q=\bigoplus_{\sigma\in I_q \setminus \{\alpha,\beta\} } K_\sigma,& & \text{for all $q$}.
					\end{align*}
					Formula \eqref{eq_newdiff} is obtained by decomposing the new differential.
				\end{proof}
				We want describe a procedure to reduce based complexes with several isomorphisms
				that fit together in a certain way. This is analogous of what happens 
				in discrete Morse theory for the reduction of CW complexes, as described by Forman \cite{For}.
				\subsection{Morse matchings}
				Let $G=(V,E)$ be a directed graph, with vertex set $V$ and edge set $E$.
				A \emph{partial matching} on $G$ is a subset $M$ of $E$ such that
				no vertex in $V$ is common to any two arrows in $M$.
				Let $G^M$ be the directed
				graph obtained by reversing the arrows in $M$. Given a partial matching $M$, 
				let also $V_M$ be the set of vertices that are neither 
				the source nor the target of any arrow in $M$. 
			
				We associate to a based complex $K^\bullet$ as above a directed graph $G_{K^\bullet}$, whose
				vertex set is the index set $V:=\cup_q I_q$ and with a directed edge 
				$\sigma\rightarrow\tau$	whenever $d_\sigma^\tau\neq 0$.
				\begin{dfn}\label{dfn_Morsematch}
					A \emph{finite Morse matching} on a based complex $K^\bullet$ is a finite 
					partial matching $M$ on $G_{K^\bullet}$ such that:
					\begin{enumerate}
						\item for all $(\sigma\rightarrow\tau)\in M$, 
							the corresponding component $d_\sigma^\tau$ is an isomorphism;
						\item the directed graph $G_{K^\bullet}^M$ 
							has no directed cycle.
					\end{enumerate}
				\end{dfn}
				The vertices $\sigma\in V_M$, and the corresponding factors $K_\sigma$, are called \emph{critical} with respect to $M$.
				Next we show that $K^\bullet$ is homotopic to a summand supported on the critical factors.
				\subsection{Repeated Gaussian Elimination}
				Let $K^\bullet$ be a based complex and let $M$ be a Morse matching on its 
				associated graph $G_{K^\bullet}=(V,E)$. 
				Let $\tilde{K}^\bullet$ be the complex with graded pieces:
				\[
					\tilde{K}^q = \bigoplus_{\sigma\in I_q \cap V_M} K_\sigma,
				\]
				and differential $\tilde{d}$, defined as follows.
				For $\sigma,\tau\in V_M$, let $\Gamma_{\sigma}^{\tau}$ be the set of 
				\emph{zigzag paths} $\gamma$ from $\sigma$ to $\tau$ the form:	
				\begin{equation}\label{eq_gamma}
					\sigma=\sigma_0\rightarrow\tau_1\leftarrow\sigma_1\rightarrow\tau_2\leftarrow\dots\rightarrow\tau_{k-1}\leftarrow\sigma_{k-1}\rightarrow\tau_{k}=\tau,				
				\end{equation}
				where the arrows are in $E$ and the leftward arrows $(\tau_i\leftarrow \sigma_i)$ are in $M$, 
				for all $i$
				(so that $\gamma$ becomes an actual directed path in $G_{K^\bullet}^M$). 
				Notice that the components of the differential corresponding to the 
				leftward arrows are invertible by definition.
				To a path $\gamma\in \Gamma_{\sigma}^\tau$ as above, we associate the following \emph{zigzag map}:
				\begin{equation}
					m(\gamma)=(-1)^{k} d_{\sigma_{k-1}}^{\tau_k} (d_{\sigma_{k-1}}^{\tau_{k-1}})^{-1}\dots d_{\sigma_1}^{\tau_2}(d_{\sigma_1}^{\tau_1})^{-1} d_{\sigma_0}^{\tau_1}
				\end{equation}
				The differential $\tilde{d}$ is given, componentwise, by:
				\begin{equation}\label{eq_tildediff}
					\tilde{d}_{\sigma}^\tau= \sum_{\gamma\in\Gamma_\sigma^\tau} m(\gamma),	
				\end{equation}
				for $\sigma,\tau\in V_M$.
				Notice that, by condition (ii) in Definition \ref{dfn_Morsematch}, the set $\Gamma_\sigma^\tau$ is finite, so the sum is well defined.
				Furthermore, notice that, along a path of the form \eqref{eq_gamma}, the cohomological 
				degree goes alternatively up and down by 1, because each arrow corresponds to a component of the differential map which has degree 1.
				So the set $\Gamma_\sigma^\tau$ is empty unless
				$\sigma\in I_q$ and $\tau\in I_{q+1}$, for some $q$. 
				Hence $\tilde{d}$ is a degree one map on $\tilde{K}^\bullet$. 
%
				\begin{thm}\label{thm_mtge}
					Endowed with $\tilde{d}$, the object $\tilde{K}^\bullet$ is a complex and a 
					Gaussian summand of $K^\bullet$.
				\end{thm}
				\begin{proof}
					We proceed by induction on the size of $M$. If $M$ is empty then there is nothing to prove.
					Otherwise let $e=(\alpha\rightarrow \beta) \in M$ be any arrow in the 
					matching. The set $M\setminus \{e\}$ is also a Morse matching. Then, by induction, the 
					complex $K^\bullet$ has a Gaussian summand $\overline{K}^\bullet$, whose associated graph $G_{\overline{K}^\bullet}$ has vertex set
					$V_M\cup \{\alpha,\beta\}$, and whose differential $\overline{d}$ is given by:
					\begin{equation}\label{eq_inddiff}
						\overline{d}_\sigma^\tau = \sum_{\gamma\in \overline{\Gamma}_\sigma^\tau} m(\gamma),
					\end{equation}
					where $\overline{\Gamma}_\sigma^\tau$ is the set of zigzag paths corresponding to $G_{\overline{K}^\bullet}$, with the 
					Morse matching $M\setminus \{e\}$.
					Observe that $\overline{\Gamma}_\alpha^\beta$ only contains the path $e=(\alpha\rightarrow \beta)$, otherwise $G_{K^\bullet}^M$ would 
					contain directed cycles and $M$ would not be a Morse matching. Then $\overline{d}_\alpha^\beta=d_\alpha^\beta$, which is an isomorphism.
					We claim that applying Lemma \ref{lem_2gauss} to this component gives the complex $\tilde{K}^\bullet$. 
		
					In fact, the new vertex set is precisely $V_M$ and the components of the new differential map are, according to \eqref{eq_newdiff} and \eqref{eq_inddiff}:
					\begin{equation}\label{eq_indstep}
						\tilde{d}_\sigma^\tau = \overline{d}_\sigma^\tau - \overline{d}_\alpha^\tau (d_\alpha^\beta)^{-1} \overline{d}_\sigma^\beta
						=\sum_{\gamma\in \overline{\Gamma}_\sigma^\tau}m(\gamma) - 
						\sum_{\gamma\in \overline{\Gamma}_\alpha^\tau}\sum_{\gamma'\in \overline{\Gamma}_\sigma^\beta}	m(\gamma)(d_\alpha^\beta)^{-1}m(\gamma').
					\end{equation}
					Observe that, for $\gamma\in \overline{\Gamma}_\alpha^\tau$ and 
					$\gamma'\in \overline{\Gamma}_\sigma^\beta$, we have:
					\[
						-m(\gamma)(d_\alpha^\beta)^{-1} m(\gamma')=m(\gamma''),
					\]
					where $\gamma''$ is the zigzag path obtained by connecting 
					$\gamma$ to $\gamma'$ via $e$. Furthermore all zigzag paths in 
					$\Gamma_\sigma^\tau$ either do not contain $e$, so belong to $\overline{\Gamma}_\sigma^\tau$, or are obtained in this way. Hence
					\eqref{eq_indstep} gives exactly \eqref{eq_tildediff}.
				\end{proof}
			\section{Reducing Rouquier Compelxes}
			Now we apply this technique to our case to obtain the reduced representatives from \S \ref{subs_reducedcom}. 
			We will describe $E^\bullet_{\uw}$ as a based complex and 
			construct a Morse matching to apply the results of the previous section.
			\subsection{Words with linkings}\label{subs_extrem}
			The first step is a convenient decomposition of the original complex.
			We will need the following notions.
			\begin{dfn}
				A \emph{repetition} in a Coxeter word $\uw\in \cword$ is a 
				length 2 subsequence consisting of equal letters.
			\end{dfn}
			\begin{exa}
				The word $\uw=sttsss$ has 3 repetitions, namely the subsequence $tt$ and the two subsequences $ss$ after it.
			\end{exa}
			\begin{dfn}
				A \emph{linking} of a Coxeter word $\uw\in \cword$ is a 
				subset of the set of repetitions.
				We denote the linking by overlining the chosen repetitions.
				The set of linkings of $\uw$ is denoted by $\mathcal{L}_{\uw}$.
			\end{dfn}
			\begin{exa}
				Here are all the possible linkings of $\uw=sttsss$:
				\begin{align*}
					& sttsss, &
					& stts\overline{ss}, &
					& stt\overline{ss}s, &
					& stt\overline{sss}, &
					& s\overline{tt}sss, &
					& s\overline{tt}s\overline{ss}, &
					& s\overline{tt}\overline{ss}s, &
					& s\overline{tt}\overline{sss}. 
				\end{align*}
				When three (or more) equal letters are overlined we 
				mean that all the repetitions that they give are in the linking.
			\end{exa}
%
			\subsection{Decomposition of Bott-Samelson objects with repetitions}\label{subs_decompos}
			Recall the decomposition of Example \ref{exa_basicdeco}. 
			Iterating, we get:
			\begin{equation}\label{eq_iteratdeco}
				B_s^{\otimes (n+1)}=\bigoplus_{k=0}^{n} B_s(n-2k)^{\oplus {\binom{n}{k}}}.			
			\end{equation}
			The inclusion morphisms of the summands are all the possible combinations of the inclusions $\iota_1$ and $\iota_2$ from
			Example \ref{exa_basicdeco}. More precisely the inclusion of one of the
			$\binom{n}{k}$ summands $B_s(n-2k)$ is of the form:
			\begin{equation}\label{eq_diagincl}
				\begin{tikzpicture}[x=.3cm,y=.7cm,baseline=.7cm]
					\draw[red] (0,0)--(0,1);
					\draw[red] (-15,2) ..controls (-15,1).. (0,1) ..controls (15,1).. (15,2);
					\clip (-15,2) ..controls (-15,1).. (0,1) ..controls (15,1).. (15,2);
					\foreach \i in {-13,-9,-7,-5,-1,1,3,7,9,11}{
						\draw[red] (\i,0)--(\i,2);
						}
					\foreach \i in {-8,0,8}{
						\node[anchor=south] at (\i,1) {$\delta_{\re{s}}$};
						}
					\foreach \i in {-11,-3,5,13}{
						\node[anchor=south,red] at (\i,1.5) {\dots};
						}
				\end{tikzpicture}
			\end{equation}
			with $n+1$ strands on top and a choice of exactly $k$ decorations $\delta_s$ in between them.
			The corresponding projection is obtained by appropriately combining projections
			$\pi_1$ and $\pi_2$ from Example \ref{exa_basicdeco}: one reflects the diagram \eqref{eq_diagincl} vertically,
			and replaces empty decorations with $\delta'_s=-s(\delta_s)$ and $\delta_s$ with empty decorations.
			
			We label these summands, and the corresponding inclusions and projections,
			by the linkings of the word $\uw=sss\dots s$. 
			For a linking $\lambda\in\mathcal{L}_{\uw}$ with $n-k$ chosen repetitions,
			we define $C_\lambda$ to be the copy of $B_s(n-2k)$ in \eqref{eq_iteratdeco} such that:
			\begin{enumerate}
				\item the corresponding inclusion has no decoration between the strands corresponding to linked letters, and the decoration
					$\delta_s$ between the strands corresponding to non-linked letters;
				\item the corresponding projection has the decoration $\delta'_s$ between the strands corresponding to
					linked letters, and no decoration between the strands corresponding to non-linked letters.
			\end{enumerate}
			\begin{pro}\label{pro_labeldecom1}
				The decomposition \eqref{eq_iteratdeco} rewrites as:
				\begin{equation}
					B_{\uw}=B_s^{\otimes(n+1)}=\bigoplus_{\lambda\in\mathcal{L}_{\uw}} C_{\lambda}.
				\end{equation}
			\end{pro}
			\begin{proof}
				It suffices to show that we have correctly labelled all the summands in \eqref{eq_iteratdeco}.
				Fix a value of $k$ between 0 and $n$. There are precisely 
				$\binom{n}{k}$ copies of $B_s(n-2k)$ in 
				\eqref{eq_iteratdeco}, each of which is determined by the 
				positions of the
				decorations in the diagram \eqref{eq_diagincl} of its inclusion. 
				The rule (i) above associates with it exactly one linking 
				on the word $\uw$ with $n-k$ repetitions, corresponding 
				to the empty decorations of the diagram.
				The linkings with $n-k$ repetitions are precisely $\binom{n}{n-k} = \binom{n}{k}$, so we have used them all.
				Of course one could instead work with the projection and use the rule (ii).
			\end{proof}
			\begin{exa}
				There are $\binom{3}{2}=3$ copies of 
				$B_s(-1)$ inside $B_s^{\otimes 4}$. Their inclusion	mophisms are:
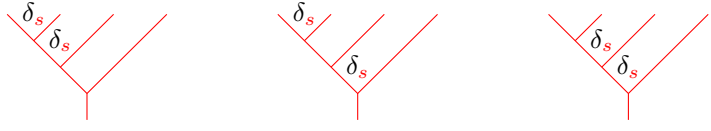
				\begin{align*}
					&\begin{tikzpicture}[scale=.35,baseline=1cm]
						\draw[red] (0,0)--(0,1)--(-3,4);
							\foreach \i in {0,1,2}{%
								\draw[red] ({-\i},{1+\i})--({3-2*\i},4);
								}
							\foreach \i in {1,2}{%
								\node[anchor=south, inner sep=.2cm] at ({-\i},{1+\i}) {$\delta_{\re{s}}$};
								}
					\end{tikzpicture}&
					&\begin{tikzpicture}[scale=.35,baseline=1cm]
						\draw[red] (0,0)--(0,1)--(-3,4);
							\foreach \i in {0,1,2}{%
								\draw[red] ({-\i},{1+\i})--({3-2*\i},4);
								}
							\foreach \i in {0,2}{%
								\node[anchor=south, inner sep=.2cm] at ({-\i},{1+\i}) {$\delta_{\re{s}}$};
								}
					\end{tikzpicture}&
					&\begin{tikzpicture}[scale=.35,baseline=1cm]
						\draw[red] (0,0)--(0,1)--(-3,4);
							\foreach \i in {0,1,2}{%
								\draw[red] ({-\i},{1+\i})--({3-2*\i},4);
								}
							\foreach \i in {0,1}{%
								\node[anchor=south, inner sep=.2cm] at ({-\i},{1+\i}) {$\delta_{\re{s}}$};
								}
					\end{tikzpicture}
				\end{align*}
				and the corresponding projection morphisms are, respectively:
				\begin{align*}
					&
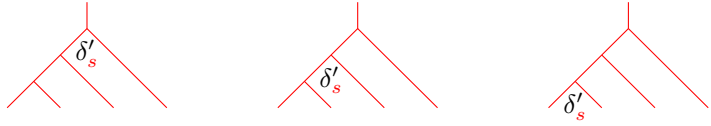
\begin{tikzpicture}[scale=.35,yscale=-1,baseline=0]
						\draw[red] (0,0)--(0,1)--(-3,4);
							\foreach \i in {0,1,2}{%
								\draw[red] ({-\i},{1+\i})--({3-2*\i},4);
								}
							\foreach \i in {0}{%
								\node[anchor=north,inner sep=0.15cm] at ({-\i},{1+\i}) {$\delta'_{\re{s}}$};
								}
					\end{tikzpicture}&
					&\begin{tikzpicture}[scale=.35,yscale=-1,baseline=0]
						\draw[red] (0,0)--(0,1)--(-3,4);
							\foreach \i in {0,1,2}{%
								\draw[red] ({-\i},{1+\i})--({3-2*\i},4);
								}
							\foreach \i in {1}{%
								\node[anchor=north,inner sep=0.15cm] at ({-\i},{1+\i}) {$\delta'_{\re{s}}$};
								}
					\end{tikzpicture}&
					&\begin{tikzpicture}[scale=.35,yscale=-1,baseline=0]
						\draw[red] (0,0)--(0,1)--(-3,4);
							\foreach \i in {0,1,2}{%
								\draw[red] ({-\i},{1+\i})--({3-2*\i},4);
								}
							\foreach \i in {2}{%
								\node[anchor=north,inner sep=0.15cm] at ({-\i},{1+\i}) {$\delta'_{\re{s}}$};
								}
					\end{tikzpicture}
				\end{align*}
				These three summands are labelled, respectively: $C_{ss\overline{ss}}$, $C_{s\overline{ss}s}$ and  $C_{\overline{ss}ss}$.				
			\end{exa}
			For an arbitrary word $\ux$ consider its decomposition \eqref{eq_mondecomp}. It is clear that:
			\begin{equation}\label{eq_multilabel}
				\mathcal{L}_{\ux} = \mathcal{L}_{\ux_1}\times \mathcal{L}_{\ux_2} \times \dots \times \mathcal{L}_{\ux_k}.
			\end{equation}
			Under this identification, for 
			$\lambda=(\lambda_1,\lambda_2,\dots,\lambda_k) \in \mathcal{L}_{\ux}$, 
			we set:
			\begin{equation}\label{eq_multiC}
			C_\lambda:=C_{\lambda_1}C_{\lambda_2}\cdots C_{\lambda_k}.			
			\end{equation}
			\begin{pro}\label{pro_multiwdec}
			Let $\ux$ be a Coxeter word, then:
			\[
				B_{\ux}=\bigoplus_{\lambda\in\mathcal{L}_{\ux}} C_{\lambda}.
			\]
			\end{pro}
			\begin{proof}
				Applying Proposition \ref{pro_labeldecom1} to each $\ux_i$ in the decomposition \eqref{eq_mondecomp}, we get:
				\[
					B_{\ux}=B_{\ux_1}\dots B_{\ux_k}=\bigg( \bigoplus_{\lambda_1\in\mathcal{L}_{\ux_1}} \!C_{\lambda_1}\bigg)\dots \bigg( \bigoplus_{\lambda_k\in\mathcal{L}_{\ux_k}}\! C_{\lambda_k}\bigg)=\bigoplus_{\lambda\in\mathcal{L}_{\ux_1}\!\times\dots\times\mathcal{L}_{\ux_k}}\!\!\!C_{\lambda_1}\dots C_{\lambda_k},
				\]
				and the last term, by \eqref{eq_multilabel} and \eqref{eq_multiC}, is the right-hand side in the statement.
			\end{proof}
			\subsection{Based Rouquier complexes}\label{subs_based}
			Consider now the graded piece $E^q_{\uw}$.
			For a subexpression $\boi$ of $\uw$, let $\uw_{\boi}$ be the subword of $\uw$ corresponding 
			to it. 
			Using \eqref{eq_grpiecE} and Proposition \ref{pro_multiwdec}, we have:
			\begin{equation}\label{eq_basedR}
				E_{\uw}^q=\bigoplus_{\substack{\boi\in \{0,1\}^{\ell(\uw)} \\ q_{\boi}=q}} B_{\boi}(q)=
				\bigoplus_{\substack{\boi\in \{0,1\}^{\ell(\uw)} \\ q_{\boi}=q}}\,\bigoplus_{\lambda\in\mathcal{L}_{\uw_{\boi}}} C_\lambda (q).
			\end{equation}
			We want to encode the linking $\lambda$ on $\uw_{\boi}$ in the subexpression $\boi$ itself.
			\begin{dfn}
				Let $\boi$ be a 01-sequence for the word $\uw$. Color the symbol of $\boi$ according to the corresponding letters in $\uw$.
				\begin{enumerate}
					\item We call \emph{repetition} in $\boi$ a subsequence of the form $10\dots 01$ where the two 1's have the same color and the symbols
						between them are all 0's.
					\item A \emph{linking} of $\boi$ is then a subset of the set of repetitions of $\boi$.
						Two 1's of $\boi$ are said to be \emph{linked} when they form a repetition appearing in the linking.
					\item A 01-sequence with a chosen linking is called a \emph{linked 01-sequence}. 
				\end{enumerate}
				We express a linking on $\boi$ by overlining the linked 1's (as well as the 0's in between).
				Let $V_{\uw}$ denote the set of linked 01-sequences of $\uw$.
			\end{dfn}
			\begin{exa}\label{exa_listsss}
				Let $\uw=sss$. The set $V_{\uw}$ is:
				\begin{equation}\label{eq_listlinksss}
					\{111,1\overline{11},\overline{11}1,\overline{111};
					011,0\overline{11}; 101,\overline{101},
					110,\overline{11}0;001;010;100;000\}.
				\end{equation}
				For instance, the expression $\overline{111}$ denotes the linked 01-sequence where both repetitions of 1 are in the linking.
			\end{exa}
			There is a clear bijection between $V_{\uw}$ and the set of pairs $(\boi,\lambda)$ with $\lambda$ a 
			linking of $\uw_{\boi}$. Under this identification let $K_\sigma:=C_{\lambda}(q_{\boi})$ and $q_\sigma:=q_{\boi}$.
			Hence \eqref{eq_basedR} can be rewritten as:
			\begin{equation}\label{eq_basedR2}
				E_{\uw}^q=\bigoplus_{\substack{\sigma\in V_{\uw} \\ q_\sigma=q}} K_{\sigma}.
			\end{equation}
			\begin{exa}\label{exa_linksss}
				For $\uw=sss$, the complex $E^\bullet_{\uw}$ is the following:
				\begin{equation}
					\begin{tikzcd}[row sep=tiny]
																	& B_{\underline{ss}}(1) \ar[r]\ar[rdd]		& B_s(2)\ar[rdd]	& \\
																	& \oplus									& \oplus			& \\
						B_{\underline{sss}} \ar[ruu]\ar[r]\ar[rdd]	& B_{\underline{ss}}(1) \ar[ruu]\ar[rdd]		& B_s(2)\ar[r]	& \un(3)\\
																	& \oplus									& \oplus			& \\
																	& B_{\underline{ss}}(1) \ar[ruu]\ar[r]		& B_s(2)\ar[ruu]	& \\
					\end{tikzcd}
				\end{equation}
				which, after decomposition becomes:
				\begin{equation}
						\begin{tikzcd}[row sep=tiny]
																								& B_s \oplus B_s(2) \ar[r]\ar[rdd]		& B_s(2) \ar[rdd]	& \\
																								& \oplus								& \oplus				& \\
							B_s(-2)\oplus B_s^{\oplus 2} \oplus B_s(2) \ar[ruu]\ar[r]\ar[rdd]	& B_s \oplus B_s(2) \ar[ruu]\ar[rdd]	& B_s(2) \ar[r]		& \un(3)\\
																								& \oplus								& \oplus				& \\
																								& B_s \oplus B_s(2) \ar[ruu]\ar[r]		& B_s(2) \ar[ruu]	& \\
						\end{tikzcd}
				\end{equation}
				or, using our notation:
				\begin{equation}\label{eq_cubemorse}
						\begin{tikzcd}[row sep=tiny]
																																						& K_{011}\oplus \textcolor{violet}{K_{0\overline{11}}} \ar[r]\ar[rdd]			& K_{001}\ar[rdd]		& \\
																																						& \oplus																		& \oplus					& \\
							K_{111} \oplus \bl{K_{1\overline{11}}}\oplus\re{K_{\overline{11}1}} \oplus \gr{K_{\overline{111}}} \ar[ruu]\ar[r]\ar[rdd]	& \re{K_{101}}\oplus \textcolor{brown}{K_{\overline{101}}} \ar[ruu]\ar[rdd]		& \textcolor{violet}{K_{010}}\ar[r]	& K_{000}\\
																																						& \oplus																		& \oplus					& \\
																																						& \bl{K_{110}}\oplus \gr{K_{\overline{11}0}} \ar[ruu]\ar[r]						& \textcolor{brown}{K_{100}}\ar[ruu]		& \\
						\end{tikzcd}
				\end{equation}
				The colors describe a Morse matching in this complex, 
				which we will define in general later.
%
			\end{exa}
			\subsection{Components of the differential}
			Now we study some properties of the differential when we take 
			the former decomposition into account. In particular, given 
			two linked 01-sequences $\sigma,\tau\in V_{\uw}$, we give 
			a necessary condition on $\sigma$ and $\tau$ 
			for $d_{\sigma}^\tau$ to be nonzero, and
			a sufficient condition for it to be an isomorphism.
			\begin{dfn}\label{dfn_cover}
				Let $\sigma,\tau\in V_{\uw}$. We say that $\sigma $ \emph{covers}
				$\tau$ if:
				\begin{enumerate}
					\item the 01-sequence of $\tau$ is obtained from that of $\sigma$ by turning a 1 into a 0;
					\item the linkings of $\sigma$ and $\tau$ contain the 
						same repetitions on the left and on the right
						of the changed symbol (but they may differ on repetitions involving 
						the eliminated 1 or going across the new 0).
					\end{enumerate}
			\end{dfn}
			\begin{exa}\label{exa_covering}
				Let $\uw=\re{s}\bl{t}\re{ss}\gr{u}\re{s}$ and consider the linked 01-sequence 
				$\sigma=\re{1}\bl{0}\overline{\re{11}\gr{0}\re{1}}$. 
				Here are the linked 01-sequences covered by $\sigma$ (and how they are obtained):
				\begin{align*}
					& \re{0}\bl{0}\overline{\re{11}\gr{0}\re{1}} & &\text{(eliminating the first 1),} \\ 
					& \re{1}\bl{0}\re{0}\overline{\re{1}\gr{0}\re{1}} & &\text{(eliminating the second 1 and its link to the third),}\\
					& \overline{\re{1}\bl{0}\re{01}\gr{0}\re{1}} & &\text{(eliminating the second 1 and adding a link across the new 0),}\\
					& \re{1}\bl{0}\re{10}\gr{0}\re{1} & &\text{(eliminating the third 1 and both its links),}\\
					& \re{1}\bl{0}\overline{\re{10}\gr{0}\re{1}} & & \text{(eliminating the third 1 and adding a link across the new 0),} \\
					& \re{1}\bl{0}\overline{\re{11}}\gr{0}\re{0}& &\text{(eliminating the last 1 and its link).}
				\end{align*}
			\end{exa}
			\begin{pro}\label{pro_covers}
				Let $\sigma,\tau\in V_{\uw}$. Then the component $d_\sigma^\tau$ of the
				differential is nonzero only if $\sigma$ covers $\tau$.
			\end{pro}
			\begin{proof}
				Let $\boi$ be the 01-sequence corresponding to $\sigma$ and
				$\boj$ the one corresponding to $\tau$.
				The component $d_\sigma^\tau$ is the composition:
				\begin{equation}\label{eq_compos}
					K_\sigma\rightarrow B_{\boi}\rightarrow B_{\boj}\rightarrow K_\tau
				\end{equation}
			where the first and last arrows are the inclusions and projections described in \S \ref{subs_decompos}, and the middle
			arrow is the differential described in \S \ref{subs_heckcat_diagrams} and \S \ref{subs_signconv}.
			If condition (i) in Definition \ref{dfn_cover}
			does not hold then the middle arrow is 0. If it holds
			then the composition has the following form:
			\begin{equation}\label{eq_diagrcompos}
				\pm\,\,
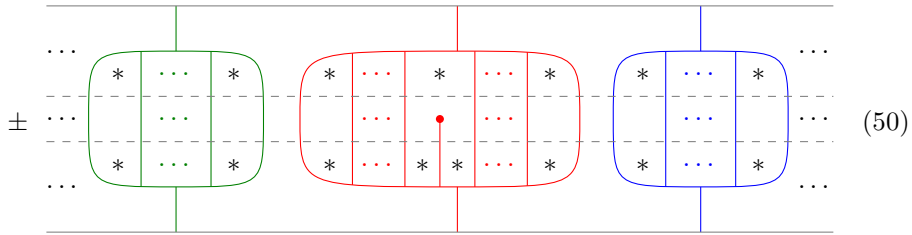
				\begin{tikzpicture}[x=.23cm,y=.6cm,baseline=1.35cm]
					\begin{scope}
						\draw[red] (0,0)--(0,1);
						\draw[red] (0,4)--(0,5);
						\draw[red] (-9,2.5) ..controls (-9,1).. (0,1) ..controls (7,1).. (7,2.5)..controls (7,4).. (0,4) ..controls (-9,4).. (-9,2.5);
						\clip (-9,2.5) ..controls (-9,1).. (0,1) ..controls (7,1).. (7,2.5)..controls (7,4).. (0,4) ..controls (-9,4).. (-9,2.5);
						\foreach \i in {-6,-3,1,4}{
							\draw[red] (\i,0)--(\i,4);
							}
						\draw[red](-1,0)--(-1,2.5); \fill[red] (-1,2.5) circle (1.5pt);
						\foreach \i in {-4.5,2.5}{
							\node[red] at (\i,1.5) {\dots};
							}		
						\foreach \i in {-4.5,2.5}{
							\node[red] at (\i,3.5) {\dots};
							}		
						\foreach \i in {-4.5,2.5}{
							\node[red] at (\i,2.5) {\dots};
							}		
						\foreach \i in {-7.3,-2,0,5.3}{
							\node at (\i,1.4) {*};
							}								
						\foreach \i in {-7.3,-1,5.3}{
							\node at (\i,3.4) {*};
							}								
					\end{scope}
					\begin{scope}[xshift=-3.7cm]	
						\draw[green] (0,0)--(0,1);
						\draw[green] (0,4)--(0,5);
						\draw[green] (-5,2.5) ..controls (-5,1).. (0,1) ..controls (5,1).. (5,2.5)..controls (5,4).. (0,4) ..controls (-5,4).. (-5,2.5);
						\clip (-5,2.5) ..controls (-5,1).. (0,1) ..controls (5,1).. (5,2.5)..controls (5,4).. (0,4) ..controls (-5,4).. (-5,2.5);
						\foreach \i in {-2,2}{
							\draw[green] (\i,0)--(\i,4);
							}
						\foreach \i in {0}{
							\node[green] at (\i,1.5) {\dots};
							}						
						\foreach \i in {0}{
							\node[green] at (\i,3.5) {\dots};
							}						
						\foreach \i in {0}{
							\node[green] at (\i,2.5) {\dots};
							}						
						\foreach \i in {-3.3,3.3}{
							\node at (\i,1.4) {*};
							}
						\foreach \i in {-3.3,3.3}{
							\node at (\i,3.4) {*};
							}
					\end{scope}
					\begin{scope}[xshift=3.2cm]	
						\draw[blue] (0,0)--(0,1);
						\draw[blue] (0,4)--(0,5);
						\draw[blue] (-5,2.5) ..controls (-5,1).. (0,1) ..controls (5,1).. (5,2.5)..controls (5,4).. (0,4) ..controls (-5,4).. (-5,2.5);
						\clip (-5,2.5) ..controls (-5,1).. (0,1) ..controls (5,1).. (5,2.5)..controls (5,4).. (0,4) ..controls (-5,4).. (-5,2.5);
						\foreach \i in {-2,2}{
							\draw[blue] (\i,0)--(\i,4);
							}
						\foreach \i in {0}{
							\node[blue] at (\i,1.5) {\dots};
							}						
						\foreach \i in {0}{
							\node[blue] at (\i,3.5) {\dots};
							}						
						\foreach \i in {0}{
							\node[blue] at (\i,2.5) {\dots};
							}						
						\foreach \i in {-3.3,3.3}{
							\node at (\i,1.4) {*};
							}
						\foreach \i in {-3.3,3.3}{
							\node at (\i,3.4) {*};
							}
					\end{scope}
					\draw[gray] (-23.5,0)--(21.5,0);\draw[gray] (-23.5,5)--(21.5,5);
					\draw[gray,dashed] (-23.5,2)--(21.5,2);\draw[gray,dashed] (-23.5,3)--(21.5,3);
					\node at (-22.5,1) {\dots};\node at (-22.5,2.5) {\dots};\node at (20.5,1) {\dots};
					\node at (-22.5,4) {\dots};\node at (20.5,2.5) {\dots};\node at (20.5,4) {\dots};
				\end{tikzpicture}			
			\end{equation}
			Here we have supposed that the changed 1 has neighboring 1's 
			of the same color on both sides: we leave to the reader the 
			picture in the other cases 
			(for example if it has no neighboring 
			1's of the same color on either side then the central part of 
			the diagram is just a dot).
			The bottom part of \eqref{eq_diagrcompos} is the inclusion 
			and the top part is the projection. The stars 
			should be replaced by (empty) decorations according to \S \ref{subs_decompos}.
			Consider each region delimited by strands of the same color and 
			containing two stars (so not the one with the dot). The two vertical strands delimiting such a 
			region correspond to 1's forming a repetition both 
			in $\sigma$ and $\tau$, which does not go across the changed symbol. If condition (ii) in Definition \ref{dfn_cover} does not hold then, 
			for at least one of these regions, we have either:
			\begin{enumerate}
				\item[(a)] the two 1's considered 
					are linked in $\sigma$ but not in $\tau$; or
				\item[(b)] the two 1's considered 
					are linked in $\tau$ but not in $\sigma$.
			\end{enumerate}
			In case (a) both stars in that region are replaced by empty decorations. In case (b)
			the bottom star is replaced by $\delta_s$ and the top one by $\delta_s'$ (for some $s$). 
			As the product $\delta_s\delta_s'$ is 
			$s$-invariant, it can be taken out of the region by Relation \eqref{eq_relsliding}, 
			so in both cases we obtain a diagram with an empty region delimited by strands of a single color. Hence the morphism is 0 by 
			Lemma \ref{lem_emptyregion}.
			\end{proof}
			\begin{dfn}\label{dfn_strcover}
				Let $\sigma,\tau\in V_{\uw}$ with $\sigma$ covering $\tau$. 
				We say that $\sigma$ \emph{strongly covers}
				$\tau$ if the 1 in $\sigma$ which is being changed appears in at least one link and furthermore:
				\begin{enumerate}
					\item if it appears in only one link, then $\tau$ does not contain 
						any link across the new 0;
					\item if it appears in two links, then $\tau$ has a link connecting the two neighboring 1's across the new 0.
				\end{enumerate}
			\end{dfn}
			In Example \ref{exa_covering}, the 01-sequences which are strongly covered by $\sigma$ are the third and the last two of the list.
			\begin{pro}\label{pro_strcover}
				Let $\sigma,\tau\in V_{\uw}$. If $\sigma$ strongly covers $\tau$ then $K_\sigma=K_\tau$
				and $d_\sigma^\tau=(-1)^a\id$, where
				$a$ is the number of 1's preceding the changed symbol.
			\end{pro}
			\begin{proof}
				Let $\sigma$ correspond to $(\boi,\lambda)$ and $\tau$ to $(\boj,\mu)$, as in \S \ref{subs_based}.
				Then $K_\sigma=C_\lambda(q_\sigma)$ and $K_\tau=C_{\mu}(q_\tau)$.
				Notice that $\sigma$ and $\tau$ only differ by one symbol and one 
				link involving it (or going across it). Hence,
				if we break $\lambda$ and $\lambda'$ into their monotonous parts, 
				as in \S \ref{subs_decompos}, 
				they will only differ in one of them:
				\begin{align}
					&\lambda=(\lambda_1,\lambda_2,\dots,\lambda_i,\dots,\lambda_k),&
					&\mu=(\lambda_1,\lambda_2,\dots,\lambda_i',\dots,\lambda_k).
				\end{align}
				We claim that $C_{\lambda_i}=C_{\lambda_i'}(-1)$.
				If $\lambda_i$ is a linking with $k$ links on the monotonous word 
				with $n$ repetitions of a certain letter $s$,
				then $\lambda_i'$ is one with $k-1$ links on a word of 
				length $n-1$, so:
				\begin{align}
					&C_{\lambda_i}=B_s(n-2k)& &C_{\lambda_i'}=B_s(n-1-2(k-1))=B_s(n-2k+1).
				\end{align}
				This proves the claim and, as $q_\sigma=q_\tau-1$, this implies $K_\sigma=K_\tau$.
				
				Now consider the diagram \eqref{eq_diagrcompos}. 
				Notice that the sign is determined only by 
				the middle arrow in \eqref{eq_compos} and it agrees with that of the statement
				by \S \ref{subs_signconv}. The number $a$ is just 
				the number of strands to the left of the dot in the middle strip of \eqref{eq_diagrcompos}.
				As $\sigma$ covers $\tau$, in all the regions containing 
				two stars, one of these stars is an empty decoration and 
				the other one is not. As $\sigma$ strongly covers $\tau$,
				Definition \ref{dfn_strcover}, translated in terms of the arrangement of decorations, implies that we have two possibilities
				in the region containing the dot:
				\begin{enumerate}
					\item only one of the two bottom stars is a $\delta_s$ (for the appropriate $s$)
						and the top star is an empty decoration;
					\item both bottom stars are empty decorations and the top star is a $\delta_s'$.
				\end{enumerate}
				So, after absorbing the dot by Relation \eqref{eq_frobunit}, the diagram becomes:
				\begin{equation}\label{eq_diagsempl}
					(-1)^a 
					\begin{tikzpicture}[x=.23cm,y=.4cm,baseline=.7cm]
						\begin{scope}
							\draw[red] (0,0)--(0,1);
							\draw[red] (0,3)--(0,4);
							\draw[red] (-9,2) ..controls (-9,1).. (0,1) ..controls (7,1).. (7,2)..controls (7,3).. (0,3) ..controls (-9,3).. (-9,2);
							\clip (-9,2) ..controls (-9,1).. (0,1) ..controls (7,1).. (7,2)..controls (7,3).. (0,3) ..controls (-9,3).. (-9,2);
							\foreach \i in {-6,-3,1,4}{
								\draw[red] (\i,0)--(\i,4);
								}
							\foreach \i in {-4.5,2.5}{
								\node[red] at (\i,2) {\dots};
								}		
							\foreach \i in {-7.3,-1,5.3}{
								\node at (\i,1.9) {*};
								}								
						\end{scope}
						\begin{scope}[xshift=-3.4cm]	
							\draw[green] (0,0)--(0,1);
							\draw[green] (0,3)--(0,4);
							\draw[green] (-5,2) ..controls (-5,1).. (0,1) ..controls (5,1).. (5,2)..controls (5,3).. (0,3) ..controls (-5,3).. (-5,2);
							\clip (-5,2) ..controls (-5,1).. (0,1) ..controls (5,1).. (5,2)..controls (5,3).. (0,3) ..controls (-5,3).. (-5,2);
							\foreach \i in {-2,2}{
								\draw[green] (\i,0)--(\i,4);
								}
							\foreach \i in {0}{
								\node[green] at (\i,2) {\dots};
								}						
							\foreach \i in {-3.3,3.3}{
								\node at (\i,1.9) {*};
								}
						\end{scope}
						\begin{scope}[xshift=2.95cm]	
							\draw[blue] (0,0)--(0,1);
							\draw[blue] (0,3)--(0,4);
							\draw[blue] (-5,2) ..controls (-5,1).. (0,1) ..controls (5,1).. (5,2)..controls (5,3).. (0,3) ..controls (-5,3).. (-5,2);
							\clip (-5,2) ..controls (-5,1).. (0,1) ..controls (5,1).. (5,2)..controls (5,3).. (0,3) ..controls (-5,3).. (-5,2);
							\foreach \i in {-2,2}{
								\draw[blue] (\i,0)--(\i,4);
								}
							\foreach \i in {0}{
								\node[blue] at (\i,2) {\dots};
								}						
							\foreach \i in {-3.3,3.3}{
								\node at (\i,1.9) {*};
								}
						\end{scope}
						\draw[gray] (-22.5,0)--(20.5,0);\draw[gray] (-22.5,4)--(20.5,4);
						\node at (-21.5,2) {\dots};\node at (19.5,2) {\dots};
					\end{tikzpicture}			
				\end{equation}
				where each star is either a $\delta$ or a $\delta'$ of the appropriate color. Notice that 
				Relations \eqref{eq_relsliding}, \eqref{eq_frobunit} and \eqref{eq_needle} imply:
				\begin{equation}\label{eq_deltaout}
					\begin{tikzpicture}[baseline=-.1cm]
						\draw[red] (0,-1)--(0,-.5);
						\draw[red] (0,0) circle (.5cm);
						\draw[red] (0,1)--(0,.5);
						\node at (0,0) {$\delta_{\re{s}}$};
					\end{tikzpicture}
					=
					\begin{tikzpicture}[baseline=-.1cm]
						\draw[red] (0,-1)--(0,-.5);
						\draw[red] (0,0) circle (.5cm);
						\draw[red] (0,1)--(0,.5);
						\node at (-1,0) {$\re{s}(\delta_{\re{s}})$};
					\end{tikzpicture}
					+
					\begin{tikzpicture}[baseline=-.1cm]
						\draw[red] (0,-1)--(0,-.5);
						\draw[red] (-120:.5cm) arc (-120:120:.5cm);
						\fill[red] (-120:.5cm) circle (1.5pt);
						\fill[red] (120:.5cm) circle (1.5pt);
						\draw[red] (0,1)--(0,.5);
						\node at (-.2,0) {$\partial_{\re{s}}(\delta_{\re{s}})$};
					\end{tikzpicture}
					=
					\quad
					\begin{tikzpicture}[baseline=-.1cm]
						\draw[red] (0,-1)--(0,1);
					\end{tikzpicture}					
				\end{equation}
				and the same holds if we replace $\delta_s$ with $\delta_s'$.
				Hence diagram \eqref{eq_diagsempl} is just the identity, up to the sign.
			\end{proof}
			\subsection{A Morse Matching}\label{subs_amorse}
			Consider the based complex $\Roudg{\uw}$ with the decomposition \eqref{eq_basedR2}.
			The associated graph $G$ has vertex set $V_{\uw}$, the set of linked 01-sequences for $\uw$.
			We want to define a Morse matching on this based complex. First we introduce the following
			notion.
			\begin{dfn}\label{dfn_Morsesymb}
				Given $\sigma\in V_{\uw}$, we call the \emph{Morse symbol} 
				of $\sigma$, if it exists, the rightmost one which is either:
				\begin{enumerate}
					\item a 0 such that the first 1 on its left has the same color; or
					\item a linked 1.
				\end{enumerate}
			\end{dfn}
			\begin{exa}
				Consider the set \eqref{eq_listlinksss} from Example \ref{exa_listsss}. We rewrite it highlighting the Morse symbols:
				\[
					\{111,1\overline{1\textbf{1}},\overline{1\textbf{1}}1,\overline{11\textbf{1}};
					011,0\overline{1\textbf{1}}; 1\textbf{0}1,\overline{10\textbf{1}},
					11\textbf{0},\overline{11}\textbf{0};001;01\textbf{0};10\textbf{0};000\}.
				\]
				So notice that 111, 011, 001, and 000 do not have symbols satisfying (i) or (ii) in Definition \ref{dfn_Morsesymb}, so they do not have a Morse symbol.
			\end{exa}
			\begin{pro}\label{pro_criticalp}
				A linked 01-sequence $\sigma\in V_{\uw}$, corresponding to $\ux\preceq\uw$, 
				has no Morse symbol if and only if it is the primary 
				01-sequence for $\ux$ with no links.
			\end{pro}
			\begin{proof}
				The linked 01-sequence $\sigma$ has no symbols of the type (ii)
				in Definition \ref{dfn_Morsesymb} if and only if
				it has no link.
				
				On the other hand, if $\sigma$ has a 0 as in (i) of Definition \ref{dfn_Morsesymb}, 
				then it is not primary: consider $\sigma'$ obtained from $\sigma$ by
				turning such a 0 to 1 and turning the first 1 on its left 
				(which has the same color) to 0. Then $\sigma'$ also 
				corresponds to $\ux$ and is smaller in the lexicographic order.
				Conversely, we prove by induction on $\ell(\uw)$
				that the 01-sequence for $\ux$ with no 0 of this type is 
				primary. The basis step is trivial.
				So let $\uw$ be of positive length and $\sigma$ a 01-sequence for $\ux$ with no 0 of this type.
				Let $\sigma_0$ be the primary 01-sequence for $\ux$. We show that $\sigma=\sigma_0$.
				Let $\uw'$ be obtained from $\uw$ by eliminating the 
				last letter, let $\sigma'$ be the 01-sequence of $\uw'$ 
				obtained from $\sigma$ by eliminating the last symbol 
				and let	$\ux'$ be the corresponding word. Let also
				$\sigma_0'$ be the primary 01-sequence for $\ux'$. 
				Notice that $\sigma_0'$ is obtained from
				$\sigma_0$ by eliminating the last symbol (because they are both minimal). 
				By induction we have $\sigma'=\sigma_0'$ so we only need to check equality on the last symbol.
				Consider the last symbol of $\sigma$:
				\begin{enumerate}
					\item If it is 1, then the last letter of $\ux$ is of that color, say red, so also the last 1 of $\sigma_0$ is red.
						If the last symbol of $\sigma_0$ was 0 then the first 1 on its left would be red, 
						but then $\sigma_0$ would not be primary 
						by the preceding part of the argument. So the last symbol of $\sigma_0$ has to be 1 too.
					\item If it is a 0, then its color is not that of the last 
						letter of $\ux$, otherwise the first 1 on its left would 
						be of this same color, contradicting the assumption. 
						Then also the last symbol of $\sigma_0$ is 0, because $\sigma_0$ also corresponds to $\ux$.
				\end{enumerate}
				This concludes the proof.
			\end{proof}
			\begin{dfn}
				Given $\sigma,\tau\in V_{\uw}$, 
				we declare that $\sigma$ is \emph{matched} to $\tau$, if $\sigma$ has a Morse symbol
				which is an overlined 1 and 
				$\tau$ is obtained from $\sigma$ by turning it to 0,
				and eliminating the link connecting it with the 
				first 1 on its left. 
			\end{dfn}
			In Example \ref{exa_covering}, the linked 01-sequence $\sigma$ is matched to the last 01-sequence of the 
			list. 
			\begin{lem}\label{lem_match}
				With $\sigma$ and $\tau$ as in the definition, the considered 0 of $\tau$
				is the Morse symbol of $\tau$, so only $\sigma$ is matched to it.
			\end{lem}
			\begin{proof}
				The linked 01-sequences $\sigma$ and $\tau$ are of the form:
				\begin{align*}
					\sigma 	&=*\dots * \overline{10\dots 0\textbf{1}} *\dots *, \\
					\tau	&=*\dots * 10\dots 0\textbf{0} *\dots *,
				\end{align*}
				where the only symbol being changed from $\sigma$ to $\tau$ is the boldfaced one.
				The boldfaced 0 in $\tau$ satisfies condition (i) in Definition \ref{dfn_Morsesymb}. If, by contradiction, 
				it is not the Morse symbol of $\tau$, 
				then the latter is on the right of it. In particular, 
				there is a symbol in $\tau$ on the right of the boldfaced 0, which satisfies (i) or (ii) of
				Definition \ref{dfn_Morsesymb}. In this case, 
				consider the corresponding symbol in $\sigma$. This also 
				satisfies one of those two conditions, and it 
				is on the right of the boldfaced 1, which is the Morse symbol of $\sigma$ by assumption. This gives a contradiction.
			\end{proof}						 
			Notice also that, if $\sigma$ is matched to $\tau$, then it strongly 
			covers it, so by Proposition \ref{pro_strcover}, the component
			$d_\sigma^\tau$ is an isomorphism.
			This, together with Lemma \ref{lem_match}, implies that the set:
			\[
				M=\{\sigma\rightarrow \tau\mid\text{$\sigma$ is matched to $\tau$}\}
			\]
			is a partial matching in $G$ and it satisfies the first condition in Definition \ref{dfn_Morsematch}.
			\begin{pro}\label{pro_morsematch}
				The partial matching $M$ is a Morse matching.
			\end{pro}
			\begin{proof}
				We only need to prove that $G^M$, the graph obtained by reversing the arrows in $M$, 
				has no directed cycle. 
				For $\sigma\in V_{\uw}$ let $p(\sigma)$ be 
				the difference between the number of 1's and the 
				number of links in $\sigma$.
				If $\sigma$ covers $\tau$, then:
				\begin{enumerate}
					\item $p(\sigma)\ge p(\tau)$; and
					\item $p(\sigma)=p(\tau)$ if and only if $\sigma$ strongly covers $\tau$.
				\end{enumerate}	
				In particular equality holds if $\sigma$ matches to $\tau$.
				Hence $p$ decreases along paths in $G^M$, so it must 
				remain constant on cycles. 
	
				Furthermore, as $M$ is a partial matching, any path cannot have two 
				consecutive arrows coming from $M$ (i.e., obtained by reversing an arrow in $M$).
				Notice that such arrows decrease the cohomological degree, 
				whereas arrows not coming from $M$ increase it. 				
				So any cycle must alternate arrows from $M$ to arrows not from $M$, and have even length.
				
				Suppose $\gamma$ is a cycle.
				If $\gamma$ is non trivial then it has to contain a path of of one of the forms:
				\begin{align}\label{eq_cycle}
					&\sigma_0\rightarrow \tau_1 \rightsquigarrow \sigma_1,&
					&\tau_1\rightsquigarrow \sigma_1\rightarrow\tau_2
				\end{align}
				where the squiggly arrows come from $M$ (they point rightwards in $G^M$, they would point leftwards in $G$).
				We claim that $\sigma_1$ is strictly less than $\sigma_0$ 
				in the lexicographic order and $\tau_2$ is strictly less than $\tau_1$. 
				This implies that $\gamma$ cannot contain such paths and therefore must be trivial.
				
				We treat the first case of \eqref{eq_cycle}, the second being similar.
				For $d_{\sigma_0}^{\tau_1}$ to be nonzero, $\sigma_0$ must cover $\tau_1$.
				Firstly, as $p(\sigma_0)=p(\tau_1)$ and $\sigma_0\rightarrow\tau_1\notin M$, the linked 01-sequence 
				$\tau_1$ must be obtained from $\sigma_0$ by turning to 0 a linked 1, which is not
				the Morse symbol of $\sigma_0$. Secondly, we know that $\sigma_1$ is matched to $\tau_1$, so it is 
				obtained from $\tau_1$ by turning to 1 the Morse symbol of the latter, which is a 0 (and adding an appropriate link). This 0 in $\tau_1$ must be to the right
				of the preceding one (from the 1 of $\sigma_0$). Composing these two operation we see that $\sigma_1$
				is obtained from $\sigma_0$ by moving a 1 to the right. 
				This decreases the lexicographic order.
			\end{proof}
			The complex \eqref{eq_cubemorse} in Example \ref{exa_linksss} 
			is colored according to the Morse matching in that case (the black factors are the critical ones with respect to it).
			\subsection{Proof of Theorem \ref{thm_reduce}}\label{subs_proof}
			We now recollect the results from the previous paragraphs and apply Morse theoretical Gaussian elimination to 
			the based complex $E^\bullet_{\uw}$, with the decomposition 
			from \S \ref{subs_based} and the Morse matching $M$ 
			from \S \ref{subs_amorse}.
			By Proposition \ref{pro_criticalp}, the critical factors with respect of 
			this Morse matching are precisely those labelled by primary 01-sequences with no links.
			Let $V_M$ be the set of these.
			By Theorem \ref{thm_mtge}, the complex $E^\bullet_{\uw}$ has a 
			Gaussian summand $\tilde{K}^\bullet$ with:
			\[
				\tilde{K}^q=\bigoplus_{\substack{\sigma\in V_M \\ q_\sigma=q}} K_\sigma,
			\]
			and differential $\tilde{d}$ given by the sums of zigzag morphisms as in \eqref{eq_tildediff}.
			
			It is clear that, for all $q$, the graded piece 
			$\tilde{K}^q$ is isomorphic to $R^q_{\uw}$: 
			it suffices to associate each subword to its primary subexpression.
			So, to conclude the proof, we need to check that they are 
			isomorphic as complexes. 
			\begin{pro}
				Let $\sigma,\tau\in V_M$ be the primary 01-sequences for
				the subwords $\ux,\uz\preceq\uw$, respectively. 
				Then, under the identification of $\tilde{K}^q$ with $R^q_{\uw}$, for all $q$, we have
				$\tilde{d}_\sigma^\tau =d^{\uz}_{\ux}$.
			\end{pro}
			\begin{proof}
				Let $\ux$ and $\uz$ be of the form \eqref{eq_formxx}.
				To describe $\tilde{d}_\sigma^\tau$ we need to find the zigzag paths from $\sigma$ to $\tau$.
				The cohomological degree of $\tau$ is one more than that of $\sigma$.
				Recall the quantity $p$ from the proof of Proposition 
				\ref{pro_morsematch}. Observe that $p(\tau)=p(\sigma)-1$, because $\sigma$ and $\tau$ 
				have no links and the latter has exactly one symbol 1 less than the former. 
				This implies that the quantity $p$ must
				decrease exactly once, by one, along any zigzag path.
				Let $\gamma$ be such a path of the form:
				\begin{equation}
					\sigma=\sigma_0\rightarrow \tau_1\leftarrow\sigma_1\rightarrow \tau_2 \leftarrow \dots \rightarrow \tau_k =\tau
				\end{equation}
				Notice that $\tau_1$ is obtained from $\sigma_0$ by turning a 1 into a 0, 
				so $p(\tau_1)\le p(\sigma_0)-1$ (because $\sigma_0$ has no links and $\tau_1$ has exactly one symbol 1 less than $\sigma_0$). 
				Then equality must hold because otherwise the quantity $p$ would decrease by more than one.
				Hence $\tau_1$ must 
				have no links. 
				Now, the quantity $p$ must remain constant on the rest of 
				the path.
				This means that all other arrows correspond to strong covering relations. 
				Then $\sigma_1$ must have exactly one link and 
				$\tau_2$ must be obtained from it by 
				turning to 0 the only linked 1 which is not the Morse 
				symbol and eliminating the link.
				So also $\tau_2$ has no link. Repeating the argument, we see that the entire path is determined
				by the first arrow $\sigma_0\rightarrow\tau_1$ and, 
				for all $i$, that $\tau_i$, $\sigma_i$ and $\tau_{i+1}$ are of the following form:
				\begin{align}\label{eq_form01seq}
					&\tau_i=\dots10\dots0\textbf{0}\dots,&
					&\sigma_i=\dots\overline{10\dots0\textbf{1}}\dots,&
					&\tau_{i+1}=\dots00\dots01\dots,
				\end{align}
				where in $\tau_i$ and $\sigma_i$ 
				the Morse symbol is boldfaced.
				In particular, the two 01-sequences $\tau_i$ and $\tau_{i+1}$ correspond to the same subword (we have just moved a 1 to another position, maintaining its color).
				Hence $\tau_1$ corresponds to the same subword as $\tau_k=\tau$, namely $\uz$.
				
				Now we are ready to compute the morphism $m(\gamma)$, starting from the first piece $d_{\sigma_0}^{\tau_1}$. 
				The 01-sequence $\tau_1$ corresponds to $\uz$ and it 
				is obtained from $\sigma=\sigma_0$ by turning a 1 to a 0.
				So the changed symbol corresponds to the letter $s$ in 
				the notation of \eqref{eq_formxx}.
				Consider diagram \eqref{eq_diagrcompos}: the morphism $d_{\sigma_0}^{\tau_1}$
				is of the following form (where $a$ is the number of strands preceding the dot in the middle strip):
			\begin{equation}\label{eq_diagA}
					(-1)^a \,\,
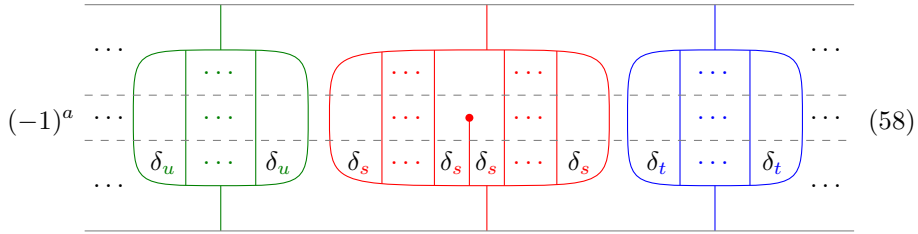
					\begin{tikzpicture}[x=.23cm,y=.6cm,baseline=1.35cm]
						\begin{scope}
							\draw[red] (0,0)--(0,1);
							\draw[red] (0,4)--(0,5);
							\draw[red] (-9,2.5) ..controls (-9,1).. (0,1) ..controls (7,1).. (7,2.5)..controls (7,4).. (0,4) ..controls (-9,4).. (-9,2.5);
							\clip (-9,2.5) ..controls (-9,1).. (0,1) ..controls (7,1).. (7,2.5)..controls (7,4).. (0,4) ..controls (-9,4).. (-9,2.5);
							\foreach \i in {-6,-3,1,4}{
								\draw[red] (\i,0)--(\i,4);
								}
							\draw[red](-1,0)--(-1,2.5); \fill[red] (-1,2.5) circle (1.5pt);
							\foreach \i in {-4.5,2.5}{
								\node[red] at (\i,1.5) {\dots};
								}		
							\foreach \i in {-4.5,2.5}{
								\node[red] at (\i,3.5) {\dots};
								}		
							\foreach \i in {-4.5,2.5}{
								\node[red] at (\i,2.5) {\dots};
								}		
							\foreach \i in {-7.3,-2,0,5.3}{
								\node at (\i,1.5) {$\delta_{\re{s}}$};
								}								
						\end{scope}
						\begin{scope}[xshift=-3.5cm]	
							\draw[green] (0,0)--(0,1);
							\draw[green] (0,4)--(0,5);
							\draw[green] (-5,2.5) ..controls (-5,1).. (0,1) ..controls (5,1).. (5,2.5)..controls (5,4).. (0,4) ..controls (-5,4).. (-5,2.5);
							\clip (-5,2.5) ..controls (-5,1).. (0,1) ..controls (5,1).. (5,2.5)..controls (5,4).. (0,4) ..controls (-5,4).. (-5,2.5);
							\foreach \i in {-2,2}{
								\draw[green] (\i,0)--(\i,4);
								}
							\foreach \i in {0}{
								\node[green] at (\i,1.5) {\dots};
								}						
							\foreach \i in {0}{
								\node[green] at (\i,3.5) {\dots};
								}						
							\foreach \i in {0}{
								\node[green] at (\i,2.5) {\dots};
								}						
							\foreach \i in {-3.3,3.3}{
								\node at (\i,1.5) {$\delta_{\gr{u}}$};
								}
						\end{scope}
						\begin{scope}[xshift=3cm]	
							\draw[blue] (0,0)--(0,1);
							\draw[blue] (0,4)--(0,5);
							\draw[blue] (-5,2.5) ..controls (-5,1).. (0,1) ..controls (5,1).. (5,2.5)..controls (5,4).. (0,4) ..controls (-5,4).. (-5,2.5);
							\clip (-5,2.5) ..controls (-5,1).. (0,1) ..controls (5,1).. (5,2.5)..controls (5,4).. (0,4) ..controls (-5,4).. (-5,2.5);
							\foreach \i in {-2,2}{
								\draw[blue] (\i,0)--(\i,4);
								}
							\foreach \i in {0}{
								\node[blue] at (\i,1.5) {\dots};
								}						
							\foreach \i in {0}{
								\node[blue] at (\i,3.5) {\dots};
								}						
							\foreach \i in {0}{
								\node[blue] at (\i,2.5) {\dots};
								}						
							\foreach \i in {-3.3,3.3}{
								\node at (\i,1.5) {$\delta_{\bl{t}}$};
								}
						\end{scope}
						\draw[gray] (-23,0)--(21,0);\draw[gray] (-23,5)--(21,5);
						\draw[gray,dashed] (-23,2)--(21,2);\draw[gray,dashed] (-23,3)--(21,3);
						\node at (-21.5,1) {\dots};
						\node at (-21.5,2.5) {\dots};
						\node at (-21.5,4) {\dots};
						\node at (19.5,1) {\dots};
						\node at (19.5,2.5) {\dots};
						\node at (19.5,4) {\dots};
					\end{tikzpicture}			
				\end{equation}				
				Or, if the dot is on the first or the last red strand, respectively of the forms:
				\begin{equation}\label{eq_diagB}
					(-1)^a \,\,
					\begin{tikzpicture}[x=.23cm,y=.6cm,baseline=1.35cm]
						\begin{scope}
							\draw[red] (0,0)--(0,1);
							\draw[red] (0,4)--(0,5);
							\draw[red] (-9,2.5) ..controls (-9,1).. (0,1) ..controls (7,1).. (7,2.5)..controls (7,4).. (0,4) ..controls (-6,4).. (-6,2.5);
							\fill[red] (-9,2.5) circle (1.5pt);
							\clip (-9,2.5) ..controls (-9,1).. (0,1) ..controls (7,1).. (7,2.5)..controls (7,4).. (0,4) ..controls (-9,4).. (-9,2.5);
							\draw[red] (-6,1)--(-6,2.5);
							\foreach \i in {-3,4}{
								\draw[red] (\i,1)--(\i,4);
								}		
							\foreach \i in {0.5}{
								\node[red] at (\i,1.5) {\dots};
								}		
							\foreach \i in {0.5}{
								\node[red] at (\i,3.5) {\dots};
								}		
							\foreach \i in {0.5}{
								\node[red] at (\i,2.5) {\dots};
								}		
							\foreach \i in {-7.1,-4.2,5.1}{
								\node at (\i,1.5) {$\delta_{\re{s}}$};
								}								
						\end{scope}
						\begin{scope}[xshift=-3.5cm]	
							\draw[green] (0,0)--(0,1);
							\draw[green] (0,4)--(0,5);
							\draw[green] (-5,2.5) ..controls (-5,1).. (0,1) ..controls (5,1).. (5,2.5)..controls (5,4).. (0,4) ..controls (-5,4).. (-5,2.5);
							\clip (-5,2.5) ..controls (-5,1).. (0,1) ..controls (5,1).. (5,2.5)..controls (5,4).. (0,4) ..controls (-5,4).. (-5,2.5);
							\foreach \i in {-2,2}{
								\draw[green] (\i,0)--(\i,4);
								}
							\foreach \i in {0}{
								\node[green] at (\i,1.5) {\dots};
								}						
							\foreach \i in {0}{
								\node[green] at (\i,3.5) {\dots};
								}						
							\foreach \i in {0}{
								\node[green] at (\i,2.5) {\dots};
								}						
							\foreach \i in {-3.3,3.3}{
								\node at (\i,1.5) {$\delta_{\gr{u}}$};
								}
						\end{scope}
						\begin{scope}[xshift=3cm]	
							\draw[blue] (0,0)--(0,1);
							\draw[blue] (0,4)--(0,5);
							\draw[blue] (-5,2.5) ..controls (-5,1).. (0,1) ..controls (5,1).. (5,2.5)..controls (5,4).. (0,4) ..controls (-5,4).. (-5,2.5);
							\clip (-5,2.5) ..controls (-5,1).. (0,1) ..controls (5,1).. (5,2.5)..controls (5,4).. (0,4) ..controls (-5,4).. (-5,2.5);
							\foreach \i in {-2,2}{
								\draw[blue] (\i,0)--(\i,4);
								}
							\foreach \i in {0}{
								\node[blue] at (\i,1.5) {\dots};
								}						
							\foreach \i in {0}{
								\node[blue] at (\i,3.5) {\dots};
								}						
							\foreach \i in {0}{
								\node[blue] at (\i,2.5) {\dots};
								}						
							\foreach \i in {-3.3,3.3}{
								\node at (\i,1.5) {$\delta_{\bl{t}}$};
								}
						\end{scope}
						\draw[gray] (-23,0)--(21,0);\draw[gray] (-23,5)--(21,5);
						\draw[gray,dashed] (-23,2)--(21,2);\draw[gray,dashed] (-23,3)--(21,3);
						\node at (-21.5,1) {\dots};
						\node at (-21.5,2.5) {\dots};
						\node at (-21.5,4) {\dots};
						\node at (19.5,1) {\dots};
						\node at (19.5,2.5) {\dots};
						\node at (19.5,4) {\dots};
					\end{tikzpicture}			
				\end{equation}				
				\begin{equation}\label{eq_diagC}
					(-1)^a \,\,
					\begin{tikzpicture}[x=.23cm,y=.6cm,baseline=1.35cm]
						\begin{scope}[xscale=-1,xshift=0.5cm]
							\draw[red] (0,0)--(0,1);
							\draw[red] (0,4)--(0,5);
							\draw[red] (-9,2.5) ..controls (-9,1).. (0,1) ..controls (7,1).. (7,2.5)..controls (7,4).. (0,4) ..controls (-6,4).. (-6,2.5);
							\fill[red] (-9,2.5) circle (1.5pt);
							\clip (-9,2.5) ..controls (-9,1).. (0,1) ..controls (7,1).. (7,2.5)..controls (7,4).. (0,4) ..controls (-9,4).. (-9,2.5);
							\draw[red] (-6,1)--(-6,2.5);
							\foreach \i in {-3,4}{
								\draw[red] (\i,1)--(\i,4);
								}		
							\foreach \i in {0.5}{
								\node[red] at (\i,1.5) {\dots};
								}		
							\foreach \i in {0.5}{
								\node[red] at (\i,3.5) {\dots};
								}		
							\foreach \i in {0.5}{
								\node[red] at (\i,2.5) {\dots};
								}		
							\foreach \i in {-7.2,-4.3,5}{
								\node at (\i,1.5) {$\delta_{\re{s}}$};
								}								
						\end{scope}
						\begin{scope}[xshift=-3.5cm]	
							\draw[green] (0,0)--(0,1);
							\draw[green] (0,4)--(0,5);
							\draw[green] (-5,2.5) ..controls (-5,1).. (0,1) ..controls (5,1).. (5,2.5)..controls (5,4).. (0,4) ..controls (-5,4).. (-5,2.5);
							\clip (-5,2.5) ..controls (-5,1).. (0,1) ..controls (5,1).. (5,2.5)..controls (5,4).. (0,4) ..controls (-5,4).. (-5,2.5);
							\foreach \i in {-2,2}{
								\draw[green] (\i,0)--(\i,4);
								}
							\foreach \i in {0}{
								\node[green] at (\i,1.5) {\dots};
								}						
							\foreach \i in {0}{
								\node[green] at (\i,3.5) {\dots};
								}						
							\foreach \i in {0}{
								\node[green] at (\i,2.5) {\dots};
								}						
							\foreach \i in {-3.3,3.3}{
								\node at (\i,1.5) {$\delta_{\gr{u}}$};
								}
						\end{scope}
						\begin{scope}[xshift=3cm]	
							\draw[blue] (0,0)--(0,1);
							\draw[blue] (0,4)--(0,5);
							\draw[blue] (-5,2.5) ..controls (-5,1).. (0,1) ..controls (5,1).. (5,2.5)..controls (5,4).. (0,4) ..controls (-5,4).. (-5,2.5);
							\clip (-5,2.5) ..controls (-5,1).. (0,1) ..controls (5,1).. (5,2.5)..controls (5,4).. (0,4) ..controls (-5,4).. (-5,2.5);
							\foreach \i in {-2,2}{
								\draw[blue] (\i,0)--(\i,4);
								}
							\foreach \i in {0}{
								\node[blue] at (\i,1.5) {\dots};
								}						
							\foreach \i in {0}{
								\node[blue] at (\i,3.5) {\dots};
								}						
							\foreach \i in {0}{
								\node[blue] at (\i,2.5) {\dots};
								}						
							\foreach \i in {-3.3,3.3}{
								\node at (\i,1.5) {$\delta_{\bl{t}}$};
								}
						\end{scope}
						\draw[gray] (-23,0)--(21,0);\draw[gray] (-23,5)--(21,5);
						\draw[gray,dashed] (-23,2)--(21,2);\draw[gray,dashed] (-23,3)--(21,3);
						\node at (-21.5,1) {\dots};
						\node at (-21.5,2.5) {\dots};
						\node at (-21.5,4) {\dots};
						\node at (19.5,1) {\dots};
						\node at (19.5,2.5) {\dots};
						\node at (19.5,4) {\dots};
					\end{tikzpicture}			
				\end{equation}				
				(Again we are assuming that there are neighboring red 
				strands, otherwise we would only have a red dot
				in the central part of the picture).
				Consider the diagrams:
				\begin{align}
					&\begin{tikzpicture}[baseline=-.1cm]
						\draw[red] (0,-1)--(0,-.5);
						\draw[red] (0,0) circle (.5cm);
						\draw[red] (0,1)--(0,.5);
						\node at (0,0) {$\delta_{\re{s}}^2$};
					\end{tikzpicture}=\quad
					\begin{tikzpicture}[baseline=-.1cm]
						\draw[red] (0,-1)--(0,1);
						\node at (1,0) {$\delta_{\re{s}}+\re{s}(\delta_{\re{s}})$};
					\end{tikzpicture},&
					&\begin{tikzpicture}[baseline=-.1cm]
						\draw[red] (0,-1)--(0,1);
						\node at (-.4,0) {$\delta_{\re{s}}$};
					\end{tikzpicture}&
					&\text{and}&
					&\begin{tikzpicture}[baseline=-.1cm]
						\draw[red] (0,-1)--(0,1);
						\node at (.5,0) {$\delta_{\re{s}}$};
					\end{tikzpicture}.
				\end{align}
				Let us call these maps $A$, $B$ and $C$ respectively.
				By retracting the dots and using \eqref{eq_deltaout},
				diagram \eqref{eq_diagA}, \eqref{eq_diagB} and \eqref{eq_diagC} become, up to sign, respectively:
				\begin{align}
					&\id_{C_{\uw_1}} A \id_{C_{\uw_2}},&
					&\id_{C_{\uw_1}} B \id_{C_{\uw_2}},&
					&\id_{C_{\uw_1}} C \id_{C_{\uw_2}}.				
				\end{align}

				We compute the remaining pieces of $m(\gamma)$. For each $i$, consider the 01-sequences \eqref{eq_form01seq}. 
				Let $a_i$ be the number of 1's preceding the Morse symbol in $\tau_i$ and $b_i$ the number of 1's preceding the other linked 1 in $\sigma_i$. Then of course
				$b_i=a_i-1$.
				By Proposition \ref{pro_strcover}, the morphism $d_{\sigma_i}^{\tau_i}$ is $(-1)^{a_i}\id$ 
				and $d_{\sigma_i}^{\tau_{i+1}}$ is $(-1)^{b_i}\id$. Hence:
				\begin{equation}\label{eq_zigzagpiece}
					-d_{\sigma_i}^{\tau_{i+1}}(d_{\sigma_i}^{\tau_i})^{-1}=(-1)^{a_i+b_i+1}\id=\id
				\end{equation}
				The zigzag morphism $m(\gamma)$ is the 
				composition of $d_{\sigma_0}^{\tau_1}$ and all the morphisms \eqref{eq_zigzagpiece}, which
				are just identities.
				Now, the choice of $\tau_1$ (and therefore of the path $\gamma$) 
				amounts to the choice of which symbol 1 in $\sigma$ 
				we turn into 0 among those corresponding the letters $s$ in \eqref{eq_formxx}, 
				or, equivalently, the choice of the red strand over which to put a dot.
				Going from left to right, the signs will alternate.
				So if $a$ is the number of strand on the left of the first $s$, we get:
				\begin{equation}\label{eq_final}
					\tilde{d}_\sigma^\tau = \sum_\gamma m(\gamma) = 
					(-1)^a \id_{C_{\uw_1}}(B - A + A - \dots +(-1)^{k-1} A -C)\id_{C_{\uw_2}},
				\end{equation}
				where $k+1$ is the number of $s$'s in $\ux$. Consider the morphism inside the parenthesis.
				If $k$ is odd, the $A$ terms cancel and we get $(B-C)$. 
				If it is even we get $(B-A+C)$. (If it is 0 then it is just the red dot).
				This is precisely the morphism $d_{s,k}$ from \S \ref{subs_reducedcom}.
				Notice that $a=\ell(\uw_1)$, so \eqref{eq_final} gives exactly $d_{\ux}^{\uz}$.
			\end{proof}
	\printbibliography[heading = bibintoc]
\end{document}